\documentclass[11pt,reqno]{amsart}

\usepackage[dvipsnames]{xcolor}
\usepackage{amsfonts}
\usepackage{amsmath,amscd}
\usepackage{amssymb}
\usepackage{tikz-cd}
\usepackage{amsthm}
\usepackage{cases}

\usepackage{galois}
\usepackage{chemarrow}
\usepackage[all]{xy}
\usepackage{graphicx}
\usepackage[colorlinks,
            linkcolor=magenta,
            anchorcolor=black,
            citecolor=red
                    ]{hyperref}	
\usepackage{mathrsfs}
\usepackage{extarrows}
\usepackage{texnames}
\usepackage{textcase}

\usepackage[deletedmarkup=xout]{changes}
\definechangesauthor[name={Per cusse}, color=OrangeRed]{.}
\allowdisplaybreaks[4]

\makeatletter
\def\section{%
  \@startsection{section}{1}
    {\z@}
    {2.0ex plus 0.8ex minus .1ex}
    {1.0ex plus .2ex}
    {\large\bfseries\boldmath\centering\MakeTextUppercase}%
}
\makeatother

\def\ve{{\rm \varepsilon}}
\newcommand{\dd}{\sqrt{-1}\partial\bar{\partial}}
\newcommand{\krf}{K\"ahler-Ricci flow\ }
\newcommand{\dl}{\left(\frac{\partial}{\partial t}-\Delta_{\omega_{\gamma\varepsilon j}(t)}\right)}
\newcommand{\osc}{\mathop{\mathrm{osc}}}
\newcommand{\tr}{\mathrm{tr}}

\newtheorem{thm}{Theorem}[section]
\newtheorem{lem}[thm]{Lemma}
\newtheorem{cor}[thm]{Corollary}
\newtheorem{rem}[thm]{Remark}
\newtheorem{pro}[thm]{Proposition}
\newtheorem{defi}[thm]{Definition}

\numberwithin{equation}{section}

\usepackage[utf8]{inputenc}

\begin{document}

\title[\sc\small L\MakeLowercase{imit behavior of the twisted conical } K\MakeLowercase{\"ahler-}R\MakeLowercase{icci flow}]{\sc\LARGE L\MakeLowercase{imit behavior of the twisted conical} K\MakeLowercase{\"ahler-}R\MakeLowercase{icci flow with change in the cone angle}}
\author[\sc\small J\MakeLowercase{iawei} L\MakeLowercase{iu and} X\MakeLowercase{i} Z\MakeLowercase{hang}]{\sc\large J\MakeLowercase{iawei} L\MakeLowercase{iu and} X\MakeLowercase{i} Z\MakeLowercase{hang}}
\address{Jiawei Liu\\School of Mathematics and Statistics\\ Nanjing University of Science \& Technology\\ Xiaolingwei Street 200\\ Nanjing 210094\\ China.} \email{jiawei.liu@njust.edu.cn}
\address{Xi Zhang\\School of Mathematics and Statistics\\ Nanjing University of Science \& Technology\\ Xiaolingwei Street 200\\ Nanjing 210094\\ China.} \email{mathzx@njust.edu.cn}
\thanks{X. Zhang is supported by National Key R and D Program of China 2020YFA0713100, NSFC (Grant Nos.12141104 and 11721101). J.W. Liu is supported by NSFC (Grant No.12371059), Jiangsu Specially-Appointed Professor Program, Fundamental Research Funds for the Central Universities, and partially supported by Special Priority Program SPP 2026 ``Geometry at Infinity" from the German Research Foundation (DFG)}
\maketitle
\vskip -3.99ex

\centerline{\noindent\mbox{\rule{3.99cm}{0.5pt}}}

\vskip 5.01ex

\ \ \ \ {\bf Abstract.} In this paper, we study the limit behavior of the conical K\"ahler-Ricci flow as its cone angle tends to zero. More precisely, we prove that as the cone angle tends to zero, the conical K\"ahler-Ricci flow converges to a unique K\"ahler-Ricci flow, which is smooth outside the divisor and admits cusp singularity along the divisor. 
\tableofcontents


\section{Introduction}

Ricci flow was first introduced by Hamilton \cite{Hamilton} in 1980s, which has many applications in differential geometry and topology. For example, Perelman \cite{Perelman1,Perelman2} gave a solution to the Poincar\'e conjecture, and Brendle-Schoen \cite{SchoenBrendle} proved the differentiable sphere theorem via Ricci flow. Then Cao \cite{Cao} introduced Ricci flow to the K\"ahler manifold, which was called K\"ahler-Ricci flow, to give a parabolic proof of the Calabi-Yau theorem. The general existence of K\"ahler-Ricci flow was given by Tsuji \cite{Tsuji} and Tian-Zhang \cite{Tianzzhang}. Later, Lott-Zhang \cite{LZ,LZ1} showed that the K\"ahler-Ricci flows can come from some singular metrics on quasi-projective manifold and that the flows keep the singularities of the initial metrics. Song-Tian \cite{JSGT} proved that the K\"ahler-Ricci flow can come from a weak metric whose potential is only continuous and then immediately smoothes it. Sz\'ekelyhidi-Tosatti \cite{GSVT} generalized this result to a more general kind of complex Monge-Amp\`ere flow, and as an application they gave the regularity result on the corresponding complex Monge-Amp\`ere equation. Moreover, Phong-Sturm \cite{PS}, Phong-Sesum-Sturm \cite{PSST}, Tian-Zhu \cite{GTXHZ07,GTXHZ13}, Phong-Song-Sturm-Weinkove \cite{PSSW0,PSSW}, Sz\'ekelyhidi \cite{Sze1}, Tosatti \cite{Tosatti}, Tian-Zhang-Zhang-Zhu \cite{TZZZ}, Song-Tian \cite{JSGT}, Berman-Boucksom-Eyssidieux-Guedj-Zeriahi \cite{BBEGZ}, Dervan-Sz\'ekelyhidi \cite{DC} and Jian-Song-Tian \cite{JST1}etc. contributed many important results on the convergence of K\"ahler-Ricci flow. In recent years, the most significant work is the solution of Hamilton-Tian's conjecture \cite{GTian}, which has been proved by Tian-Zhang \cite{Tianzhang}, Bamler \cite{Bamler}, Chen-Wang \cite{Chenwangb} and Wang-Zhu \cite{WaZhu}.
 
 As a generalization of the K\"ahler-Ricci flow, the twisted K\"ahler-Ricci flow which adds a smooth twisted form in K\"ahler-Ricci flow was first studied by Collins-Sz\'ekelyhidi \cite{TC} and the first author \cite{JWL}. By using some new arguments, Guedj-Zeriahi \cite{VGAZ2} and Di Nezza-Lu \cite{NL2017} gave the best regularity and uniqueness results of this flow starting from weak metric which only admits zero Lelong number. Recently, Zhang \cite{KWZ} proved Tian's partial $C^0$-estimate along this flow, and the authors \cite{JWLXZ3} showed the relation between the convergence of  these flows and the greatest Ricci lower bound.

When adding a non-smooth twisted form in K\"ahler-Ricci flow, more precisely, the twisted form is only a current of the integration along a divisor with positive coefficient less than $1$, such flow was called conical K\"ahler-Ricci flow which was first introduced by Jeffres-Mazzeo-Rubinstein \cite{JMR}. Then Chen-Wang \cite{CW,CW1}, the authors \cite{JWLXZ}, Wang \cite{YQW} and Shen \cite{LMSH1} studied its existence when is comes from a class of conical K\"ahler metrics or a specific one. Later, the authors \cite{JWLXZ1} and the first author and Zhang \cite{JWLCJZ} extended the initial metic to the weak one as in \cite{JSGT, GSVT}, and Li-Shen-Zheng \cite{LSZ} extended this result to the weak initial metric which only admits $L^\infty$-potential on log canonical pairs. Recently, the authors and Zhang \cite{JWLXZ4} studied this flow coming from the weakest initial metric whose Lelong number is zero, which generalized the results on smooth K\"ahler-Ricci flow proved by Guedj-Zeriahi \cite{VGAZ2} and Di Nezza-Lu \cite{NL2017} to the conical case. For the convergence and geometric properties along this flow, see \cite{GEDWA1, GEDWA, Nomura, LMSH2, YSZ, YSZ1, JWLXZ2} for more details.

The conical K\"ahler-Ricci flow was introduced to study the conical K\"ahler-Einstein metric which plays an important role in the solution of Yau-Tian-Donaldson conjecture. Such metrics exist for all small cone angles, which was conjectured by Donaldson \cite{SD2009} and proved by Berman \cite{RB} and Li-Sun \cite{LS} (see also \cite{JWL1, JWLXZ2}). So how large the cone angle can be, and what the limit behavior of this metric should be as its cone angle tends to the critical value are important and interesting problems. On the first question, if it can be related to some geometric invariant, Song-Wang \cite{SW} did important work is this subject. And from the proof of Yau-Tian-Donaldson conjecture \cite{T1, CDS1, CDS2, CDS3}, if the Fano manifold is $K$-stable, then the conical K\"ahler-Einstein metric converges to the smooth K\"ahler-Einstein metric as its cone angle tends to $2\pi$. On the other hand, Tian \cite{Tian94} anticipated that the complete Tian-Yau K\"ahler-Einstein metric on the complement of a divisor should be the limit of the conical K\"ahler-Einstein metric as cone angle tend to zero, and Cheltsov-Rubinstein \cite{ChRu15} gave a more general conjecture. Guenancia \cite{G11} and Rubinstein-Zhang \cite{RuZha1, RuZha2} contributed many important work in this research. Especially, by introducing new ideas and techniques, Biquard-Guenancia \cite{BiGu} creatively solved Tian's conjecture in Fano and some twisted canonical line bundles being ample cases. 

For the conical K\"ahler-Ricci flow itself, it exits for all cone angles belonging to $0$ to $2\pi$, and it converges to the smooth K\"ahler-Ricci flow as its cone angle tends to $2\pi$ because there is a positive lower bound on the cone angel in this case and so the uniform estimates can be obtained directly (see \cite{JWLXZ, YQW, LMSH1}). A natural question is what the limit behavior of the conical K\"ahler-Ricci flow should be as its cone angle tends to zero. As we know, only the authors \cite{JWLXZ19} proved that the conical K\"ahler-Ricci flow with a specific initial metric admitting conical singularity converges to a cusp K\"ahler-Ricci flow with a cusp K\"ahler metric as its initial metric as the cone angle tends to zero if the twisted canonical line bundle $K_X+L_D$ is ample. The key point is that the authors observe that the singularity term can be combined with the metric potential to be viewed as a whole, which can be controlled uniformly when they choose a specific initial metric. But unfortunately, this method does not work in general case. 

In this paper, we study the limit behavior of the twisted conical K\"ahler-Ricci flow as its cone angle tends to zero in general case. Let $(X, \omega)$ be a compact K\"ahler manifold with complex dimension $n$ and $D$ be a smooth divisor. Assume that $\hat\omega\in [\omega]$ is a positive closed current with potential $\varphi_0\in L^\infty(X)\cap {\rm PSH(X,\omega)}$ and that $\eta$ is a smooth closed $(1,1)$-form. The twisted conical K\"ahler-Ricci flow with initial metric $\hat\omega$ is given as
\begin{equation}\label{CK}
\left\{
\begin{aligned}
 &\ \frac{\partial \omega_{\gamma}(t)}{\partial t}=-{\rm Ric}(\omega_{\gamma}(t))+(1-\gamma)[D]+\eta\\
 &\ \ \ \ \ \ \ \ \ \ \ \ \ \ \ \ \ \ \ \ \ \ \ \ \ \ \ \ \ \ \ \ \ \ \ \ \ \ \ \ \ \ \ \ \ \ \ \ \ \ \ \ \ \ \ ,\\
 &\ \omega_{\gamma}(t)|_{t=0}=\hat\omega
\end{aligned}
\right.\tag{$CKRF^{\eta}_{\gamma}$}
\end{equation}
where $[D]$ is the current of integration along $D$. Denote 
\begin{equation}\label{0131001}
T^{\eta}_{\gamma,\max}=\sup\left\{t\geqslant0\Big| [\omega]+t\left(-c_1(X)+(1-\gamma)c_1(L_D)+[\eta]\right)\geqslant0\right\},
\end{equation}
where $L_D$ is the line bundle associated with divisor $D$. From Li-Shen-Zheng's \cite{LSZ} result (see the authors' \cite{JWLXZ4} result for more general case), we see that there exists a unique solution to the twisted K\"ahler-Ricci flow \eqref{CK} on $[0,T^\eta_{\gamma,\max})\times X$ in the sense of Definition \ref{0416002}.

By introducing some new ideas, here we give the following characterization of the limit behavior of the twisted conical K\"ahler-Ricci flow as its cone angle tends to zero.
\begin{thm}\label{0416001}
As the cone angle tends to $0$, the twisted conical K\"ahler-Ricci flow \eqref{CK} converges locally in the smooth sense outside the divisor and globally in the sense of currents to a K\"ahler-Ricci flow
\begin{equation}\label{PK}
\left\{
\begin{aligned}
 &\ \frac{\partial \omega(t)}{\partial t}=-{\rm Ric}(\omega(t))+[D]+\eta\\
 &\ \ \ \ \ \ \ \ \ \ \ \ \ \ \ \ \ \ \ \ \ \ \ \ \ \ \ \ \ \ \ \ \ \ \ \ \ \ \ \ , \\
 &\ \omega(t)|_{t=0}=\hat\omega
\end{aligned}
\right.\tag{$CKRF^{\eta}_{0}$}
\end{equation}
which exists on $[0,T^\eta_{\max})\times X$ in the following sense. 
\begin{itemize}
  \item For any $0<T<T^\eta_{\max}$, there exists a constant $C$ such that for any $t\in(0,T)$,
\begin{equation}\label{0513001}
e^{-\frac{C}{t^2}}\omega_o\leqslant\omega(t)\leqslant e^{\frac{C}{t^2}}\omega_o\ \ \ \text{in}\ \ \ X\setminus D;
\end{equation}
   \item  In $(0,T^\eta_{\max})\times(X\setminus D)$, $\omega(t)$ satisfies the smooth twisted K\"ahler-Ricci flow
 \begin{equation}\label{0513002}
\frac{\partial \omega(t)}{\partial t}=-{\rm Ric}(\omega(t))+\eta;
\end{equation}
  \item On $(0,T^\eta_{\max})\times X$, $\omega(t)$ satisfies K\"ahler-Ricci flow \eqref{PK} in the sense of currents;
  \item There exists a metric potential $\varphi(t)\in\mathcal{C}^{\infty}\left((0,T^\eta_{\max})\times(X\setminus D)\right)$ with respect to $\omega_t$ such that $\omega(t)=\omega_{t}+\dd\varphi(t)$,
  \begin{equation}\label{0513003}
  \lim\limits_{t\rightarrow0^{+}}\|\varphi(t)-\varphi_0\|_{L^{1}(X)}=0\ \ \ \text{\rm and}\ \ \ \lim\limits_{t\rightarrow0^{+}}\varphi(t)=\varphi_0\ \ \ \text{in}\ \ \ X\setminus D;
\end{equation}
\item On $[0,T^\eta_{\max})$, $\|\varphi(t)-t\psi_o\|_{L^\infty(X)}\leqslant C$;
\item For any $0<T<T^\eta_{\max}$, there exist constants $a$ and $C$ such that for any $t\in(0,T)$,
\begin{equation}\label{0513004}
-C+a\log t\leqslant\dot\varphi(t)-\psi_o\leqslant\frac{C}{t};
\end{equation}
\end{itemize}
where $\omega_o$ and $\psi_o$ are a cusp K\"ahler metric and its metric potential coming from \eqref{0328001} and \eqref{0321016} respectively, $\omega_t$ comes from \eqref{0320010}, and
\begin{equation}\label{0513005}
T^{\eta}_{\max}=\sup\left\{t\geqslant0\Big| [\omega]+t\left(-c_1(X)+c_1(L_D)+[\eta]\right)\geqslant0\right\}.
\end{equation}
\end{thm}

As is well known, the key steps in showing the limit behavior of the twisted conical K\"ahler-Ricci flow \eqref{CK} as cone angle tends to zero are to derive uniform estimates which do not depend on $\gamma$ on this flow. But it seems not to be easy. A conical K\"ahler metric with cone angle $2\pi\gamma$ along divisor is $L^{p_\gamma}$-integrable with $p_\gamma\in\left(1,\frac{1}{1-\gamma}\right)$, so for any fix $\gamma>0$, we can use the classical results in partial differential equations, such as Ko{\l}odziej's $L^{p}$-estimate and stability theorem  \cite{K000, K} etc, to derive the estimates on the twisted conical K\"ahler-Ricci flow \eqref{CK} and its smooth approximation sequence (see \cite{JWLXZ, JWLXZ1, YQW, LMSH1} for more details). But all above estimates depend on $\gamma$ and will converge to $\infty$ as $\gamma$ tends to zero. In addition, since there is no available theory for complex Monge-Amp\`ere equation with $L^1$ right hand side, and also $\omega_\gamma(t)$ may converge to a non-integrable current, there is no direct results that can be applied to derive the uniform estimates (independent of $\gamma$) on the twisted conical K\"ahler-Ricci flow \eqref{CK} with sufficiently small positive $\gamma$. Therefore, we need some new arguments to prove the uniform estimates on the twisted conical K\"ahler-Ricci flow when we consider its limit behavior .

In order to overcome above difficulties, we introduce some new ideas to derive the uniform estimates on the twisted conical K\"ahler-Ricci flow \eqref{CK} with sufficiently small positive $\gamma$. In this process, a special function (see \eqref{0320019}), thanks to Guenancia \cite{G11}, which is a metric potential for conical K\"ahler metric and converges to the potential of a cusp K\"ahler metric, plays an important role. More precisely, by using Guenancia's function, we construct a special super-solution to derive a uniform upper bound on the potential of $\omega_\gamma(t)$. Then we obtain the uniform lower bound on the potential in a short time by writing the parabolic Monge-Amp\`ere equation as an elliptic one and by proving a uniform upper bound on the time derivative of potential. At last, by constructing a new auxiliary function, we successively obtain the uniform lower bound on the time derivative of potential, which implies a uniform lower bound on the potential for any time. Combining all these estimates together, we prove the uniform Laplacian $\mathcal{C}^2$-estimate by introducing a new auxiliary function again, and then we show the convergence of the twisted conical K\"ahler-Ricci flow in the subsequence sense as the cone angle tends to zero. Moreover, since the limit behavior of the twisted conical K\"ahler-Ricci flow near the initial time depends on $\gamma$, we can not let $\gamma$ tends to zero directly when we consider the limit behavior of the limit flow near the initial time. Our method of studying this is to construct sub-solution and super-solution and then study the limit behavior of these solutions. Furthermore, by drawing inspiration from Di Nezza-Lu's \cite{NL2017} uniqueness theorem on smooth K\"ahler-Ricci flow with weak initial metric (see also the author's work \cite{JWLXZ4} for the conical case), we prove a uniqueness theorem for the K\"ahler-Ricci flow with cusp singularity along the divisor $D$, which implies that the limit flow of above convergence is unique.

As a corollary of Theorem \ref{0416001}, we give the following existence and uniqueness result.
\begin{cor}\label{0419001}
There exists a unique weak solution to the twisted cusp K\"ahler-Ricci flow on $[0,T^\eta_{\max})\times X$ in the sense of Theorem \ref{0416001}. 
\end{cor}
The smooth K\"ahler-Ricci flow with singularity along a divisor in quasi-projective manifold $X\setminus D$ was first studied by Lott-Zhang \cite{LZ,LZ1}. They thoroughly studied the existence and convergence of K\"ahler-Ricci flow whose initial metric is finite volume K\"ahler metric with ``superstandard spatial asymptotics" (Definition $8.10$ in \cite{LZ}), and they gave some interesting examples. Their flow keeps ``superstandard spatial asymptotics" and this type of metrics contain cusp K\"ahler metrics. Then the authors \cite{JWLXZ19} studied the K\"ahler-Ricci flow with cusp singularity along a divisor on $X$ globally by approximating a sequence of conical K\"ahler-Ricci flows as their cone angles tend to zero in the case that the twisted canonical line bundle $K_X+L_D$ is ample, and later Chau-Li-Shen \cite{CLS} generalized Lott-Zhang's \cite{LZ} results to more general initial metric, which means that the smooth K\"ahler-Ricci flow with cusp singularity along a divisor in $X\setminus D$ can start from a metric without cusp singularity. Recently, Li-Shen-Zheng \cite{LSZ} constructed a maximal solution (in the sense of potential) to the K\"ahler-Ricci flow \eqref{PK} on $[0,T^\eta_{\max})\times X$ in the sense of Definition \ref{0510002}. However it is not clear whether their maximal solution is the unique one to the K\"ahler-Ricci flow \eqref{PK}. Here we prove that the solution is indeed unique.

The paper is organized as follows. We first give the definition and some known results of the twisted conical K\"ahler-Ricci flow, and we prove some maximum principles in section \ref{Preliminaries}. Then we prove the uniform estimates including the $\mathcal{C}^0$-estimates on metric potentials and its time derivative in section \ref{Bou1}. In section \ref{Bou2}, we prove the uniform Laplacian $\mathcal{C}^2$ and high order estimates. We then show the limit behavior of the conical K\"ahler-Ricci flow as the cone angle tends to zero, and we give the existence of the limit flow in section \ref{Exi}.  At the last section, we prove a uniqueness theorem for the K\"ahler-Ricci flow with cusp singularity along the divisor.

\section{Preliminaries}\label{Preliminaries}

In this section, we recall some definitions and results on the twisted conical K\"ahler-Ricci flow, and then we prove some maximum principles for the equations with conical singularity. 

Let $(X, \omega)$ be a compact K\"ahler manifold with complex dimension $n$ and $D$ be a smooth divisor with definition section $s$, that is, $D=\{s=0\}$. Assume that $\hat\omega\in [\omega]$ is a positive closed current with potential $\varphi_0\in L^\infty(X)\cap {\rm PSH(X,\omega)}$, that is,
\begin{equation}\label{0320001}
\hat\omega=\omega+\dd \varphi_0.
\end{equation} 
The twisted conical K\"ahler-Ricci flow with twisted form $\eta$ which is a smooth closed $(1,1)$-form is as follows
\begin{equation}\label{CK}
\left\{
\begin{aligned}
 &\ \frac{\partial \omega_{\gamma}(t)}{\partial t}=-{\rm Ric}(\omega_{\gamma}(t))+(1-\gamma)[D]+\eta\\
 &\ \ \ \ \ \ \ \ \ \ \ \ \ \ \ \ \ \ \ \ \ \ \ \ \ \ \ \ \ \ \ \ \ \ \ \ \ \ \ \ \ \ \ \ \ \ \ \ \ \ \ \ \ \ \ \ .\\
 &\ \omega_{\gamma}(t)|_{t=0}=\hat\omega
\end{aligned}
\right.\tag{$CKRF^{\eta}_{\gamma}$}
\end{equation}
Denote 
\begin{equation}\label{0320001}
T^{\eta}_{\max}=\sup\left\{t\geqslant0\Big| [\omega]+t\left(-c_1(X)+c_1(L_D)+[\eta]\right)\geqslant0\right\},
\end{equation}
where $L_D$ is the line bundle associated with divisor $D$. 

For any $T<T^\eta_{\max}$, there exists a smooth volume form $\Omega$ such that
\begin{equation}\label{0320003}
\omega_{t}=\omega+t\left(-{\rm Ric}(\Omega)+\theta+\eta\right)>0
\end{equation}
for all $t\in[0,T]$, where $\theta$ is the curvature of a smooth Hermitian metric $h$ on $L_D$. In fact, since 
\begin{equation}\label{0320004}
[\omega]+T(-c_1(X)+c_1(L_D)+[\eta])>0,
\end{equation}
there exists a smooth positive closed $(1,1)$-form
\begin{equation}\label{0320005}
\kappa\in[\omega]+T(-c_1(X)+c_1(L_D)+[\eta]).
\end{equation}
For any $t\in[0,T]$, we denote
\begin{equation}\label{0320006}
\omega_{t}=\omega+\frac{t}{T}\left(\kappa-\omega\right)=(1-\frac{t}{T})\omega+\frac{t}{T}\kappa.
\end{equation}
It is obvious that $\omega_t$ is positive and 
\begin{equation}\label{0320007}
\omega_{t}\in[\omega]+t\left(-c_1(X)+c_1(L_D)+[\eta]\right).
\end{equation}

On the other hand, since
\begin{equation}\label{0320008}
-\frac{1}{T}\left(\kappa-\omega\right)+\theta+\eta\in c_1(X)
\end{equation} 
is a smooth closed $(1,1)$-form, by Calabi-Yau theorem, there exists a smooth volume form $\Omega$ such that
\begin{equation}\label{0320009}
\rm Ric(\Omega)=-\frac{1}{T}\left(\kappa-\omega\right)+\theta+\eta,
\end{equation}
which implies that for any $t\in[0,T]$,
\begin{equation}\label{0320010}
\omega_{t}=\omega+t\left(-{\rm Ric}(\Omega)+\theta+\eta\right).
\end{equation}
It is obvious that there is a constant $C>1$ depending only on $T$, $\theta$, $\eta$ and $\omega$ such that
\begin{equation}\label{0320011}
\frac{1}{C}\omega\leqslant\omega_{t}\leqslant C\omega
\end{equation}
for all $t\in[0,T]$. We denote 
\begin{equation}\label{0320012}
\Omega=e^{-h}\omega^n\ \ \ \ \ \text{and}\ \ \ \ \ \omega_{t}=\omega+t \nu,
\end{equation}
where 
\begin{equation}\label{0320014}
\nu=-{\rm Ric}(\Omega)+\theta+\eta.
\end{equation}
Then we can write the flow \eqref{PK} as the following parabolic Monge-Amp\`ere equation
\begin{equation}\label{PC}
\left\{
\begin{aligned}
 &\ \frac{\partial \varphi(t)}{\partial t}=\log\frac{\left(\omega_{t}+\dd\varphi(t)\right)^n}{\omega^n}+h+\log|s|_h^2\\
 &\ \ \ \ \ \ \ \ \ \ \ \ \ \ \ \ \ \ \ \ \ \ \ \ \ \ \ \ \ \ \ \ \ \ \ \ \ \ \ \ \ \ \ \ \ \ \ \ \ \ \ \ \ \ \ \ \ \ \ \ \ \ \ .\\
 &\ \varphi(t)|_{t=0}=\varphi_0
\end{aligned}
\right.
\end{equation}
We denote
\begin{equation}\label{0320002}
T^{\eta}_{\gamma, \max}=\sup\left\{t\geqslant0\Big| [\omega]+t\left(-c_1(X)+(1-\gamma)c_1(L_D)+[\eta]\right)\geqslant0\right\}.
\end{equation}
For any $T<T^\eta_{\max}$, there holds $T<T^\eta_{\gamma, \max}$ and
\begin{equation}\label{0320013}
\omega_{\gamma t}:=\omega_t-\gamma t\theta=\omega+t\left(-{\rm Ric}(\Omega)+(1-\gamma)\theta+\eta\right)>0
\end{equation}
for all $t\in[0,T]$ and $0<\gamma\ll1$ sufficiently small. From \eqref{0320011}, there also exists a large constant $C>1$ depending only on $\omega$, $T$, $\theta$ and $\eta$ such that 
\begin{equation}\label{0320016}
\frac{1}{C}\omega\leqslant\omega_{\gamma t}\leqslant C\omega
\end{equation}
for all $t\in[0,T]$ and $0<\gamma\ll1$. If we denote
\begin{equation}\label{0320015}
\nu_\gamma=-{\rm Ric}(\Omega)+(1-\gamma)\theta+\eta=\nu-\gamma\theta,
\end{equation}
then 
\begin{equation}\label{0320017}
\omega_{\gamma t}=\omega+t\nu_\gamma,
\end{equation}
and the twisted conical \krf \eqref{CK} can be written as equation
\begin{equation}\label{CC}
\left\{
\begin{aligned}
 &\ \frac{\partial \varphi_{\gamma}(t)}{\partial t}=\log\frac{\left(\omega_{\gamma t}+\dd\varphi_{\gamma}(t)\right)^n}{\omega^n}+h+(1-\gamma)\log|s|_h^2\\
 &\ \ \ \ \ \ \ \ \ \ \ \ \ \ \ \ \ \ \ \ \ \ \ \ \ \ \ \ \ \ \ \ \ \ \ \ \ \ \ \ \ \ \ \ \ \ \ \ \ \ \ \ \ \ \ \ \ \ \ \ \ \ \ \ \ \ \ \ \ \ \ \ \ \ \ \ \ .\\
 &\ \varphi_{\gamma}(t)|_{t=0}=\varphi_0
\end{aligned}
\right.
\end{equation}
\begin{defi}\label{0220003}
By saying a closed positive $(1,1)$-current $\omega$ with locally bounded potential is a conical K\"ahler metric with cone angle $2\pi \gamma $ ($0<\gamma\leqslant1$) along $D$, we mean that $\omega$ is a smooth K\"ahler metric in $M\setminus D$. And near each point $p\in D$, there exists a local holomorphic coordinate $(z_{1}, \cdots, z_{n})$ in a neighborhood $U$ of $p$ such that $D=\{z_{n}=0\}$ and that $\omega$ is asymptotically equivalent to the model conical metric
\begin{equation}\label{0220004}
 \sum_{j=1}^{n-1} \sqrt{-1}dz_{j}\wedge d\overline{z}_{j}+ \frac{\sqrt{-1} dz_{n}\wedge d\overline{z}_{n}}{|z_{n}|^{2(1-\gamma)}} \ \ \ on\ \ U.\end{equation}
\end{defi}
In this paper, we always choose Guenancia's metric (Lemma $3.1$ in \cite{G11})
\begin{equation}\label{0320018}
\begin{aligned}
\omega_\gamma&=\omega+\dd\psi_\gamma\\
&=\omega+\frac{\gamma^2\left\langle D's,D's\right\rangle}{|s|_h^{2(1-\gamma)}\left(1-|s|_h^{2\gamma}\right)^2}-\frac{\gamma|s|_h^{2\gamma}}{1-|s|_h^{2\gamma}}\theta
\end{aligned}
\end{equation}
as the model conical K\"ahler metric, where 
\begin{equation}\label{0320019}
\psi_\gamma=-\log\left(\frac{1-|s|_h^{2\gamma}}{\gamma}\right)^2
\end{equation}
and $D'$ is the $(1,0)$ part of the Chern connection of the line bundle $L_D$ with Hermitian metric $h$. Since function $f(t)=\frac{\gamma t^\gamma}{1-t^\gamma}$ increases with respect to $t\in[0,1)$, after rescaling $h$ such that $|s|_h^2<\delta$, we have
\begin{equation}\label{0320020}
0\leqslant\frac{\gamma|s|_h^{2\gamma}}{1-|s|_h^{2\gamma}}\leqslant\frac{\gamma\delta^{\gamma}}{1-\delta^{\gamma}}\to-\frac{1}{\log\delta}\ \ \ \text{as}\ \gamma\to0,
\end{equation}
which implies that 
\begin{equation}\label{0320021}
0\leqslant\frac{\gamma|s|_h^{2\gamma}}{1-|s|_h^{2\gamma}}\leqslant-\frac{2}{\log\delta}\ \ \ \ \text{for}\ \ \ 0<\gamma\ll1,
\end{equation}
and so
\begin{equation}\label{0320022}
\dd\psi_\gamma=\frac{\gamma^2\left\langle D's,D's\right\rangle}{|s|_h^{2(1-\gamma)}\left(1-|s|_h^{2\gamma}\right)^2}-\frac{\gamma|s|_h^{2\gamma}}{1-|s|_h^{2\gamma}}\theta\geqslant-\frac{\gamma|s|_h^{2\gamma}}{1-|s|_h^{2\gamma}}\theta\geqslant\frac{2C}{\log\delta}\omega
\end{equation}
for $0<\gamma\ll1$, where $C$ is a positive constant satisfying $\theta\leqslant C\omega$. Therefore, if we choose $0<\delta\ll1$ such that
\begin{equation}\label{0320023}
\dd\psi_\gamma\geqslant-\frac{1}{8\left(C-1\right)}\omega,
\end{equation}
where $C$ comes from \eqref{0320016}, then there holds
\begin{equation}\label{0320024}
\omega_\gamma=\omega+\dd\psi_\gamma\geqslant\left(1-\frac{1}{8\left(C-1\right)}\right)\omega\geqslant\frac{1}{2}\omega
\end{equation}
for $0<\gamma\ll1$. From Guenancia's arguments in \cite{G11}, we know that there exists a uniform constant $C$ independent of $\gamma$ such that
\begin{equation}\label{03200241}
\left|\log\frac{\omega_\gamma^n}{\omega^n}+(1-\gamma)\log|s|_h^2-\psi_\gamma\right|\leqslant C.
\end{equation}
Moreover, from the authors' observation (Proposition $2.8$ in \cite{JWLXZ19}), $\psi_\gamma$ is monotone decreasing to $\psi_o$ as $\gamma\searrow0^+$. Hence $\omega_\gamma$ converges in $\mathcal{C}^\infty_{loc}$-topology in $ X\setminus D$ to a K\"ahler metric
\begin{equation}\label{0328001}
\begin{aligned}
\omega_o&=\omega+\dd\psi_o\\
&=\omega+\frac{\left\langle D's,D's\right\rangle}{|s|_h^{2}\log^2|s|_h^2}+\frac{1}{\log|s|_h^2}\theta
\end{aligned}\end{equation}
as $\gamma$ tends to $0$, and
\begin{equation}\label{0321016}
-\log\log^2|s|_h^2=\psi_o\leqslant\psi_\gamma\leqslant\psi_1=-\log\left(1-|s|_h^2\right)^2.
\end{equation}
So there exists a uniform constant $C$ depending only on $\omega$, $\theta$ and $T$ such that
\begin{equation}\label{0321014}
\psi_\gamma\leqslant C\ \ \ \text{on}\ \ \ X.
\end{equation}
\begin{rem}\label{0509004}
In fact, $\omega_o$ is a cusp K\"ahler metric in the following sense.
\end{rem}
\begin{defi}\label{0509003}
By saying a closed positive $(1,1)$-current $\omega$ is a cusp K\"ahler metric along $D$, we mean that $\omega$ is a smooth K\"ahler metric in $M\setminus D$. And near each point $p\in D$, there exists a local holomorphic coordinate $(z_{1}, \cdots, z_{n})$ in a neighborhood $U$ of $p$ such that $D=\{z_{n}=0\}$ and that $\omega$ is asymptotically equivalent to the model cusp metric
\begin{equation}\label{0509002}
\sqrt{-1} \sum_{j=1}^{n-1} dz_{j}\wedge d\overline{z}_{j}+\frac{\sqrt{-1}dz_{n}\wedge d\overline{z}_{n}}{|z_{n}|^{2}\log^2|z_{n}|^{2}}\ \ \ \text{on}\ \ \ U.
\end{equation}
\end{defi}
Combining \eqref{0320023}, \eqref{0320024} with \eqref{0320016}, we also have
\begin{equation}\label{0320025}
\omega_{\gamma t}+\dd\psi_{\gamma}\leqslant C\omega+\dd\psi_\gamma=\left(C-1\right)\omega+\omega_\gamma\leqslant\left(2C-1\right)\omega_\gamma
\end{equation}
and
\begin{equation}\label{0320026}
\begin{split}
\omega_{\gamma t}+\dd\psi_{\gamma}&\geqslant \frac{1}{C}\omega+\dd\psi_\gamma=\frac{1}{C}\omega_\gamma+\left(1-\frac{1}{C}\right)\dd\psi_\gamma\\
&\geqslant\frac{1}{2C}\omega_\gamma+\frac{1}{4C}\omega+\left(1-\frac{1}{C}\right)\dd\psi_\gamma\\
&\geqslant\frac{1}{2C}\omega_\gamma+\frac{1}{4C}\omega-\left(1-\frac{1}{C}\right)\frac{1}{8\left(C-1\right)}\omega\\
&=\frac{1}{2C}\omega_\gamma+\frac{1}{8C}\omega.
\end{split}
\end{equation}
Therefore, there exists a uniform constant $C$ depending only on $\omega$, $\theta$, $\eta$ and $T$ such that for any $0<\gamma\ll1$, we have
\begin{equation}\label{0320027}
\frac{1}{C}\omega_\gamma\leqslant\omega_{\gamma t}+\dd\psi_\gamma\leqslant C\omega_\gamma\ \ \ \text{in}\ \ \ X\setminus D.
\end{equation}
\begin{defi}\label{0416002}
Assume that $\hat\omega\in [\omega]$ is a positive closed current with potential $\varphi_0\in L^\infty(X)$ (or more general, $\hat\omega\in[\omega]$ is a positive closed current with zero Lelong number). By saying $\omega_\gamma(t)$ is a weak solution to the twisted conical K\"ahler-Ricci flow \eqref{CK} with initial metric $\hat\omega$ on $(0,T')\times X$, we mean that it satisfies the following conditions.
\begin{itemize}
  \item  For any $0<\delta<T<T'$, there exists a constant $C$ such that for any $t\in(\delta,T)$,
\begin{equation}
C^{-1}\omega_\gamma\leqslant\omega_\gamma(t)\leqslant C\omega_\gamma\ \ \ in\ \ \ X\setminus D;
\end{equation}
   \item  In $(0,T')\times(X\setminus D)$, $\omega_\gamma(t)$ satisfies the smooth twisted K\"ahler-Ricci flow
  \begin{equation}\label{0221001}
\frac{\partial \omega_{\gamma}(t)}{\partial t}=-{\rm Ric}(\omega_{\gamma}(t))+\eta;
\end{equation}
  \item On $(0,T')\times X$, $\omega_\gamma(t)$ satisfies equation \eqref{CK} in the sense of currents;
  \item There exists a metric potential $\varphi_\gamma(t)\in \mathcal{C}^{2,\alpha,\gamma}\left( X\right)\cap \mathcal{C}^{\infty}\left(X\setminus D\right)$ such that $\omega_\gamma(t)=\omega_{\gamma t}+\sqrt{-1}\partial\bar{\partial}\varphi_\gamma(t)$, 
  \begin{equation}\label{0507002}
\lim\limits_{t\rightarrow0^{+}}\|\varphi_\gamma(t)-\varphi_{0}\|_{L^{1}(X)}=0\ \ \ \text{and}\ \ \ \lim\limits_{t\rightarrow0^{+}}\varphi_\gamma(t)=\varphi_{0}\ \ \text{in}\ \ X\setminus D.
\end{equation}
  \end{itemize}
\end{defi}
\begin{rem}\label{0515001}
By saying a function $\phi\in \mathcal{C}^{2,\alpha,\gamma}\left( X\right)$, we mean that $\phi$ is a $\mathcal{C}^{2,\alpha}$-function with respect to $\omega_\gamma$.
\end{rem}
\begin{defi}\label{0221001002}
We call $\phi(t)\in \mathcal{C}^{2,\alpha,\gamma}\left( X\right)\cap \mathcal{C}^{\infty}\left(X\setminus D\right)$ a weak solution to equation \eqref{CC} with initial value $\phi_0\in L^\infty(X)\cap{\rm PSH}(X,\omega)$ on $(0,T')\times X$ if $\phi(t)$ satisfies the following conditions.
\begin{itemize}
  \item  For $0<\delta<T<T'$, there exists a constant $C$ such that for any $t\in(\delta,T)$,
\begin{equation}\label{0515002}
C^{-1}\omega_\gamma\leqslant\omega_{\gamma t}+\sqrt{-1}\partial\bar{\partial}\phi(t)\leqslant C\omega_\gamma\ \ \ \text{in}\ \ \ X\setminus D;
\end{equation}
   \item  In $(0,T')\times(X\setminus D)$, $\phi(t)$ satisfies equation
 \begin{equation}\label{0515003}
\frac{\partial \phi(t)}{\partial t}=\log\frac{\left(\omega_{\gamma t}+\dd \phi(t)\right)^n}{\omega^n}+h+(1-\gamma)\log|s|_h^2;
\end{equation}
  \item $\lim\limits_{t\rightarrow0^{+}}\|\phi(t)-\phi_{0}\|_{L^{1}(X)}=0$ and $\lim\limits_{t\rightarrow0^{+}}\phi(t)=\phi_{0}$ in $X\setminus D$.
  \end{itemize}
\end{defi}
Since there are two singular terms in the twisted conical \krf \eqref{CK}, which are the initial metic $\hat\omega$ and current term $[D]$. By Demailly's regularization result \cite{JPD}, the authors \cite{JWLXZ4} choose the sequence $\omega_j$ approximating to $\hat\omega$, that is, there exists a sequence of smooth strictly $\omega$-psh functions $\varphi_{j}$ decreasing to $\varphi_0$ as $j\nearrow\infty$, and 
\begin{equation}\label{0220002}
\omega_j=\omega+\dd\varphi_j.
\end{equation}
For the current term $[D]$, as in \cite{JWLXZ}, we approximate it by $\theta_\ve=\theta+\dd\log\left(\ve^2+|s|_h^2\right)$ as $\ve\searrow0$. Then the flow \eqref{CK} is approximated by the smooth twisted K\"ahler-Ricci flow
\begin{equation}\label{TK}
\left\{
\begin{aligned}
&\ \frac{\partial \omega_{\gamma\ve j}(t)}{\partial t}=-{\rm Ric}(\omega_{\gamma\ve j}(t))+(1-\gamma)\theta_{\ve}+\eta\\
 &\ \ \ \ \ \ \ \ \ \ \ \ \ \ \ \ \ \ \ \ \ \ \ \ \ \ \ \ \ \ \ \ \ \ \ \ \ \ \ \ \ \ \ \ \ \ \ \ \ \ \ \ \ \ \ \ \ \ \ .\\
 &\ \omega_{\gamma\ve j}(t)|_{t=0}=\omega_j
\end{aligned}
\right.\tag{$TKRF^{\eta}_{\gamma\ve j}$}
\end{equation}
which is equivalent to the following complex Monge-Amp\`ere equation
\begin{equation}\label{TC}
\left\{
\begin{aligned}
 &\ \frac{\partial \varphi_{\gamma\ve j}(t)}{\partial t}=\log\frac{\left(\omega_{\gamma t}+\dd\varphi_{\gamma\ve j}(t)\right)^n}{\omega^n}+h+(1-\gamma)\log\left(\ve^2+|s|_h^2\right)\\
 &\ \ \ \ \ \ \ \ \ \ \ \ \ \ \ \ \ \ \ \ \ \ \ \ \ \ \ \ \ \ \ \ \ \ \ \ \ \ \ \ \ \ \ \ \ \ \ \ \ \ \ \ \ \ \ \ \ \ \ \ \ \ \ \ \ \ \ \ \ \ \ \ \ \ \ \ \ \ \ \ \ \ \ \ \ \ \ \ \ \ \ \ \ .\\
 &\ \varphi_{\gamma\ve j}(t)|_{t=0}=\varphi_{j}
\end{aligned}
\right.
\end{equation}
In \cite{JWLXZ4}, by using smooth approximation and the ideas of Guedj-Zeriahi \cite{VGAZ2} and Di Nezza-Lu's \cite{NL2017} regularizing properties for the smooth twisted K\"ahler-Ricci flow coming from a positive current with zero Lelong number, the authors and Zhang proved the following result.
\begin{thm}\label{0510001}
Let $\hat\omega\in[\omega]$ be a closed positive current with zero Lelong number. For any $\gamma\in(0,1]$ and smooth closed $(1,1)$-form $\eta$, there exists a unique weak solution to the twisted conical K\"ahler-Ricci flow \eqref{CK} on $(0,T^\eta_{\gamma,\max})\times X$ in the sense of Definition \ref{0416002}.
\end{thm}
\begin{rem}\label{0502001}
From Definition \ref{0221001002}, if $\phi(t)$ is a weak solution to equation \eqref{CC} with initial value $\phi_0\in L^\infty(X)\cap{\rm PSH}(X,\omega)$ on $(0,T')\times X$, for any $0<\delta<T<T'$, there exists a constant $C$ such that
\begin{equation}\label{031800399}
C^{-1}\omega_\gamma\leqslant\omega_{\gamma t}+\sqrt{-1}\partial\bar{\partial}\phi(t)\leqslant C\omega_\gamma.
\end{equation}
Hence for fix $\gamma\in(0,1)$,
\begin{equation}\label{031800499}
\frac{\partial \phi(t)}{\partial t}=\log\frac{\left(\omega_{\gamma t}+\dd\phi(t)\right)^n}{\omega^n}+h+(1-\gamma)\log|s|_h^2
\end{equation}
is bounded in $[\delta,T]\times (X\setminus D)$, which implies that $\phi(t)$ is H\"older continuous on $X$ for any $t\in[\delta,T]$ by Ko{\l}odziej's $L^p$-estimates and that
\begin{equation}\label{031800599}
\|\phi(t)-\phi(s)\|_{L^\infty(X)}\leqslant C|t-s|\ \ \ \text{for all}\ t,s\in[\delta,T].
\end{equation}
Therefore, for any $t>0$,  $\phi(t)$ is continuous on $X$, and there holds
\begin{equation}\label{0502002}
\lim\limits_{s\to t^+}\|\phi(s)-\phi(t)\|_{L^\infty(X)}=0.
\end{equation}
\end{rem}
If the metric potential of the initial metric is continuous, the only difference is that the solution of \eqref{CC} converges to the initial one in $L^\infty$-sense, which is stronger than the last condition in Definition \ref{0416002}. In \cite{JWLXZ1}, the authors gave the following existence and uniqueness theorem.
\begin{thm}[Theorem $1.2$ and Remark $1.3$ in \cite{JWLXZ19}]\label{0313000}
Assume that $\tilde\omega$ is a closed positive current with continuous metric potential $\tilde\varphi_0$ with respect to $\omega$. Then there exists a unique solution to the twisted conical K\"ahler-Ricci flow 
 \begin{equation}\label{0401001}
\left\{
\begin{aligned}
 &\ \frac{\partial \omega_{\gamma}(t)}{\partial t}=-{\rm Ric}(\omega_{\gamma}(t))+(1-\gamma)[D]+\eta\\
 &\ \ \ \ \ \ \ \ \ \ \ \ \ \ \ \ \ \ \ \ \ \ \ \ \ \ \ \ \ \ \ \ \ \ \ \ \ \ \ \ \ \ \ \ \ \ \ \ \ \ \ \\
 &\ \omega_{\gamma}(t)|_{t=0}=\tilde{\omega}
\end{aligned}
\right.
\end{equation}
on $(0,T^\eta_{\gamma,\max})\times X$ in the following sense
\begin{itemize}
  \item  For any $0<\delta<T<T^\eta_{\gamma,\max}$, there exists a constant $C$ such that for any $t\in(\delta,T)$,
\begin{equation}
C^{-1}\omega_\gamma\leqslant\omega_\gamma(t)\leqslant C\omega_\gamma\ \ \ \text{in}\ \ \ X\setminus D;
\end{equation}
   \item  In $(0,T^\eta_{\gamma,\max})\times(X\setminus D)$, $\omega_\gamma(t)$ satisfies the smooth twisted K\"ahler-Ricci flow
 \begin{equation}\label{031300211}
\frac{\partial \omega_{\gamma}(t)}{\partial t}=-{\rm Ric}(\omega_{\gamma}(t))+\eta;
\end{equation}
  \item On $(0,T^\eta_{\gamma,\max})\times X$, $\omega_\gamma(t)$ satisfies equation \eqref{0401001} in the sense of currents;
  \item There exists a metric potential $\varphi_\gamma(t)\in \mathcal{C}^{2,\alpha,\gamma}\left( X\right)\cap \mathcal{C}^{\infty}\left(X\setminus D\right)$ such that $\omega_\gamma(t)=\omega_{\gamma t}+\sqrt{-1}\partial\bar{\partial}\varphi_\gamma(t)$ and $\lim\limits_{t\rightarrow0^{+}}\|\varphi_\gamma(t)-\tilde\varphi\|_{L^{\infty}(X)}=0$;
  \end{itemize}
\end{thm}
\begin{defi}\label{0313007}
Let $u_0\in \mathcal{C}^0(X)$ be a $\omega$-psh function. By saying $u(t)\in \mathcal{C}^{2,\alpha,\gamma}\left( X\right)\cap \mathcal{C}^{\infty}\left(X\setminus D\right)$ is a sub-solution to equation \eqref{CC} with initial value $u_0$ on $(0,T')\times X$, we mean that $u(t)$ satisfies the following conditions.
\begin{itemize}
  \item  For $0<\delta<T<T'$, there exists a constant $C$ such that for any $t\in(\delta,T)$,
\begin{equation}
C^{-1}\omega_\gamma\leqslant\omega_{\gamma t}+\sqrt{-1}\partial\bar{\partial}u(t)\leqslant C\omega_\gamma\ \ \ \text{in}\ \ \ X\setminus D;
\end{equation}
   \item  In $(0,T')\times(X\setminus D)$, $u(t)$ satisfies equation
 \begin{equation}\label{0313008}
\frac{\partial u(t)}{\partial t}\leqslant\log\frac{\left(\omega_{\gamma t}+\dd u(t)\right)^n}{\omega^n}+h+(1-\gamma)\log|s|_h^2;
\end{equation}
  \item $\lim\limits_{t\rightarrow0^{+}}\|u(t)-u_0\|_{L^{\infty}(X)}=0$.
  \end{itemize}
\end{defi}
\begin{defi}\label{0313009}
Let $v_0\in \mathcal{C}^0(X)$ be a $\omega$-psh function. By saying $v(t)\in \mathcal{C}^{2,\alpha,\gamma}\left( X\right)\cap \mathcal{C}^{\infty}\left(X\setminus D\right)$ is a super-solution to equation \eqref{CC} with initial value $v_0$ on $(0,T')\times X$, we mean that $v(t)$ satisfies the following conditions.
\begin{itemize}
  \item  For $0<\delta<T<T'$, there exists a constant $C$ such that for any $t\in(\delta,T)$,
\begin{equation}
C^{-1}\omega_\gamma\leqslant\omega_{\gamma t}+\sqrt{-1}\partial\bar{\partial}v(t)\leqslant C\omega_\gamma\ \ \ \text{in}\ \ \ X\setminus D;
\end{equation}
   \item  In $(0,T')\times(X\setminus D)$, $v(t)$ satisfies equation
 \begin{equation}\label{0313010}
\frac{\partial v(t)}{\partial t}\geqslant\log\frac{\left(\omega_{\gamma t}+\dd v(t)\right)^n}{\omega^n}+h+(1-\gamma)\log|s|_h^2;
\end{equation}
  \item $\lim\limits_{t\rightarrow0^{+}}\|v(t)-v_0\|_{L^{\infty}(X)}=0$.
  \end{itemize}
\end{defi}
\begin{defi}\label{03130090}
Let $\phi_0\in \mathcal{C}^0(X)$ be a $\omega$-psh function. By saying $\phi(t)\in \mathcal{C}^{2,\alpha,\gamma}\left( X\right)\cap \mathcal{C}^{\infty}\left(X\setminus D\right)$ is a solution to equation \eqref{CC} with initial value $\phi_0$ on $(0,T')\times X$, we mean that $\phi(t)$ is both a sub-solution and a super-solution.
\end{defi}
\begin{rem}\label{0515006}
From Definitions \ref{0221001002} and \ref{03130090} and Remark \ref{0502001}, we see that a weak solution $\phi(t)$ to equation \eqref{CC} with initial value $\phi_0\in L^\infty(X)\cap{\rm PSH}(X,\omega)$ on $(0,T')\times X$ is a solution to equation \eqref{CC} with initial value $\phi(t_0)\in\mathcal{C}^0(X)\cap{\rm PSH}(X,\omega_{\gamma t_0})$ on $(t_0,T')\times X$.
\end{rem}
Next, we prove some maximum principles, which are important in proving the uniqueness theorem of the limit flow.
\begin{pro}\label{0313003}
If $u(t)$ is a sub-solution to equation \eqref{CC} with initial value $u_0\in \mathcal{C}^0(X)\cap {\rm PSH}(X,\omega)$ on $(0,T')\times X$, and $v(t)$ is a super-solution to equation \eqref{CC} with initial value $v_0\in \mathcal{C}^0(X)\cap {\rm PSH}(X,\omega)$ on $(0,T')\times X$. Then 
\begin{equation}\label{0313006}
u(t)-v(t)\leqslant\sup\limits_X\left(u_0-v_0\right)\ \ \ \text{on}\ \ \ [0,T')\times X.
\end{equation}
\end{pro}
\begin{proof}
For any $0<t_1<T<T'$ and $a>0$. We consider $\Psi(t)=u(t)-v(t)+a\log|s|_h^2$ on $[t_1,T]\times X$. If we denote
\begin{equation}\label{0314001}
\hat{\Delta}=\int_0^1 g_{s}^{i\bar{j}}\frac{\partial^2}{\partial z^i\partial\bar{z}^j}ds,
\end{equation}
where $g_s$ is the metric corresponding to $\omega_s=\omega_{\gamma t}+s\dd u(t)+(1-s)\dd v(t)$. Then $\Psi(t)$ evolves along the following equation
\begin{equation}\label{0314002}
  \frac{\partial \Psi(t)}{\partial t}\leqslant\hat{\Delta}\Psi(t)-a\hat{\Delta}\log|s|_h^{2}.
\end{equation}
Since $\omega_{u(t)}$ and $\omega_{v(t)}$ are bounded from below by $c_0\omega_\gamma$ and \eqref{0320024}, we obtain
\begin{equation}\label{0314005}
\omega_s\geqslant c\omega
\end{equation}
for some constants $c_0$ and $c$ depending on $t_1$ and $T$.
Combining this with $-\sqrt{-1}\partial\bar{\partial}\log|s|_h^2=\theta$, we conclude that there exists a constant $C(t_1,T)$ depending on $t_1$ and $T$ such that
\begin{equation}\label{0314003}
-\hat{\Delta}\log|s|_h^2=\int_0^1\tr_{\omega_s}\theta ds\leq C(t_1,T)
\end{equation}
in $X\setminus D$. Then 
\begin{equation}\label{0314004}
  \frac{\partial \Psi(t)}{\partial t}\leqslant\hat{\Delta}\Psi(t)+aC(t_1,T).
\end{equation}

Let $\tilde{\Psi}=\Psi-aC(t_1,T)(t-t_1)-\epsilon (t-t_1)$. It is obvious that the space maximum of  $\Psi(t)$ on $[t_1,T]\times X$ is attained away from $D$.  Let $(t_0,x_0)$ be the maximum point. If $t_0>t_1$, by the maximum principle, at $(t_0,x_0)$, we have
\begin{equation}\label{0314006}
 0\leqslant \Big(\frac{\partial }{\partial t}-\hat{\Delta}\Big)\tilde{\Psi}(t)\leqslant -\epsilon,
\end{equation}
which is impossible and so $t_0=t_1$. Then for $(t,x)\in [t_1,T]\times \left(X\setminus D\right)$, we obtain
\begin{equation}\label{0314007}
u(t)-v(t)\leqslant \sup\limits_X\left(u(t_1)-v(t_1)\right)+aTC(t_1,T)+\epsilon T-a\log|s|_h^2.
\end{equation}
We let $\epsilon\rightarrow0$ and then $a\rightarrow0$, there holds
\begin{equation}\label{0314008}
u(t)-v(t)\leqslant \sup\limits_X\left(u(t_1)-v(t_1)\right).
\end{equation}
Then we prove \eqref{0313006} after we let $t_1\rightarrow0^+$ since $\lim\limits_{t\rightarrow0^+}\|u(t)-u_0\|_{L^{\infty}(X)}=0$ and $\lim\limits_{t\rightarrow0^{+}}\|v(t)-v_0\|_{L^{\infty}(X)}=0$
\end{proof}
\begin{pro}\label{0318001}
Assume that $\phi(t)$ is a super-solution to equation \eqref{CC} with initial value $\phi_0\in \mathcal{C}^0(X)\cap {\rm PSH}(X,\omega)$ and that $\psi(t)$ is a weak solution to equation \eqref{CC} with initial value $\psi_0\in L^\infty(X)\cap{\rm PSH}(X,\omega)$ on $(0,T')\times X$. Then 
\begin{equation}\label{0318002}
\psi(t)-\phi(t)\leqslant\sup\limits_X\left(\psi_0-\phi_0\right)\ \ \ \text{on}\ \ \ [0,T')\times X.
\end{equation}
\end{pro}
\begin{proof}
For any $0<t_1<T<T'$, since $\psi(t)$ is a weak solution to equation \eqref{CC}, by Remark \ref{0515006}, $\psi(t)$ is a solution to equation \eqref{CC} with initial value $\psi(t_1)$. By using Proposition \ref{0313003}, on $[t_1,T')\times X$, we have
\begin{equation}\label{0418006}
\psi(t)-\phi(t)\leqslant\sup\limits_X\left(\psi(t_1)-\phi(t_1)\right).
\end{equation}
Since $\psi(t_1)$ converges to $\psi_0$ in $L^1$-sense and $\phi(t_1)$ converges to $\phi_0$ in $L^\infty$-sense, it then follows form Hartogs' Lemma that 
\begin{equation}\label{0418006}
\psi(t)-\phi(t)\leqslant\sup\limits_X\left(\psi_0-\phi_0\right)
\end{equation}
on $[0,T')\times X$ after we let $t_1\to0$.
\end{proof}

\section{Uniform estimates on metric potentials}\label{Bou1}

In this section, we prove some uniform estimates on potential $\varphi_\gamma(t)$ and its time derivative $\dot{\varphi}_\gamma(t)$.

For any $0<T<T^\eta_{\max}$, $T<T^\eta_{\gamma,\max}$ also holds for $0<\gamma\ll1$ sufficiently small. If $T>1$, we further assume that $\delta$ in \eqref{0320022} is sufficiently small such that
\begin{equation}\label{0320028}
\dd\psi_\gamma\geqslant-\frac{1}{8T\left(C-1\right)}\omega>-\frac{1}{8\left(C-1\right)}\omega,
\end{equation}
where $C$ comes from \eqref{0320016}. Combining \eqref{0320023} with \eqref{0320024}, we also have
\begin{equation}\label{03200281}
\omega_\gamma=\omega+\dd\psi_\gamma\geqslant\frac{1}{2}\omega.
\end{equation}

Due to the singularity of the equation \eqref{CC}, we are unable to derive an effective uniform estimate on $\varphi_\gamma(t)$ by using the maximum principle. But fortunately, by constructing new auxiliary functions and introducing new arguments, we can prove a uniform estimate on $\varphi_\gamma(t)-t\psi_\gamma$ and its time derivative, which enables us to derive the high order uniform estimates along the twisted conical K\"ahler-Ricci flow. Now we first construct a super-solution to derive a uniform upper bound. 
\begin{pro}\label{0320029}
There exists a uniform constant $C$ depending only on $\omega$, $\theta$, $\eta$ and $T$ such that
\begin{equation}\label{0320030}
\varphi_\gamma(t)\leqslant\sup\limits_X\varphi_0+t\psi_\gamma+Ct\ \ \ \text{on}\ \ \ X
\end{equation}
for all $t\in[0,T]$ and $0<\gamma\ll1$.
\end{pro}
\begin{proof}
Denote
\begin{equation}\label{0320031}
\psi_\gamma(t)=\sup\limits_X\varphi_0+t\psi_\gamma+Ct.
\end{equation}
By using \eqref{0320016} and \eqref{0320028}, for any $t\in[0,T]$, we have
\begin{equation}\label{0320032}
\omega_{\gamma t}+\dd\psi_\gamma(t)=\omega_{\gamma t}+t\dd\psi_\gamma\geqslant\frac{1}{C}\omega-\frac{t}{8T\left(C-1\right)}\omega\geqslant\left(\frac{1}{C}-\frac{1}{8\left(C-1\right)}\right)\omega>0,
\end{equation}
which implies that $\psi_\gamma(t)$ is a continuous $\omega_{\gamma t}$-psh function.

On the other hand, from \eqref{0320016}, \eqref{0320025} and \eqref{03200281}, there exists a uniform constant $C$ depending only on $\omega$, $\theta$, $\eta$ and $T$ such that
\begin{equation}\label{0320033}
\omega_{\gamma t}+\dd\psi_\gamma(t)=t\left(\omega_{\gamma t}+\dd\psi_{\gamma}\right)+\left(1-t\right)\omega_{\gamma t}\leqslant C\omega_\gamma.
\end{equation}
Hence 
\begin{equation}\label{0320034}
\begin{split}
&\ \ \ \log\frac{\left(\omega_{\gamma t}+\dd\psi_{\gamma}(t)\right)^n}{\omega^n}+h+(1-\gamma)\log|s|_h^2\\
&\leqslant\log C+\log\frac{\omega_\gamma^n}{\omega^n}+h+(1-\gamma)\log|s|_h^2\\
&\leqslant\log C+\log\frac{\omega_\gamma^n}{\omega^n}+h+(1-\gamma)\log|s|_h^2-\psi_\gamma+\psi_\gamma\\
&\leqslant C+\psi_\gamma=\frac{\partial\psi_\gamma(t)}{\partial t},
\end{split}
\end{equation}
where we use \eqref{03200241} in the last inequality. Combining the fact that 
\begin{equation}\label{0320034}
\lim\limits_{t\to0^+}\|\psi_\gamma(t)-\sup\limits_X\varphi_0\|_{L^\infty(X)}=0,
\end{equation}
we show that $\psi_\gamma(t)$ is a super-solution to equation \label{0320032} with initial value $\sup\limits_X\varphi_0$. By the maximum principle Proposition \ref{0318001}, we conclude that 
\begin{equation}\label{0320035}
\varphi_\gamma(t)\leqslant\psi_\gamma(t)=\sup\limits_X\varphi_0+t\psi_\gamma+Ct\ \ \ \text{on}\ \ \ X
\end{equation}
for all $t\in[0,T]$ and $0<\gamma\ll1$.
\end{proof}
\begin{pro}\label{0321001}
For any $t\in(0,T]$ and $0<\gamma\ll1$, there holds
\begin{equation}\label{0321002}
\dot\varphi_\gamma(t)\leqslant\frac{\varphi_\gamma(t)-\varphi_0}{t}+n\ \ \ \text{in}\ \ \ X\setminus D.
\end{equation}
Furthermore, there exists a uniform constant $C$ depending only on $n$, $\omega$, $\theta$, $\eta$ and $T$ such that
\begin{equation}\label{0321003}
\dot\varphi_\gamma(t)\leqslant\frac{\osc\varphi_0}{t}+\psi_\gamma+C\ \ \ \text{in}\ \ \ X\setminus D
\end{equation}
for all $t\in[0,T]$ and $0<\gamma\ll1$.
\end{pro}
\begin{proof}\label{0321004}
We consider the smooth complex Monge-Amp\`ere equation \eqref{TC}. Let
\begin{equation}\label{0321005}
H_{\gamma \ve j}(t)=t\dot\varphi_{\gamma \ve j}(t)-\left(\varphi_{\gamma\ve j }(t)-\varphi_j\right)-nt.
\end{equation}
Direct computations show that
\begin{equation}\label{0321006}
\frac{\partial}{\partial t} \dot{\varphi}_{\gamma\ve j}(t)=\tr_{\omega_{\gamma\ve j}(t)}\left(v_{\gamma}+\dd\dot{\varphi}_{\gamma\ve j}(t)\right),
\end{equation}
and hence
\begin{equation}\label{0321007}
\frac{\partial H_{\gamma \ve j}}{\partial t}=t\ddot{\varphi}_{\gamma\ve j}(t)-n=t\left(\tr_{\omega_{\gamma\ve j}(t)}\nu_\gamma+\Delta_{\omega_{\gamma\ve j}(t)}\dot{\varphi}_{\gamma\ve j}(t)\right)-n.
\end{equation}
On the other hand, we have
\begin{equation}\label{0321008}
\begin{aligned}
\Delta_{\omega_{\gamma\ve j}(t)} H_{\gamma \ve j}&=t\Delta_{\omega_{\gamma\ve j}(t)} \dot{\varphi}_{\gamma\ve j}(t)-\Delta_{\omega_{\gamma\ve j}(t)} \varphi_{\gamma\ve j}(t)+\Delta_{\omega_{\gamma\ve j}(t)} \varphi_{j}\\
&=t\Delta_{\omega_{\gamma\ve j}(t)} \dot{\varphi}_{\gamma\ve j}(t)-n+\tr_{\omega_{\gamma\ve j}(t)}\omega_{\gamma t}+\Delta_{\omega_{\gamma\ve j}(t)} \varphi_{j}\\
&=t\left(\tr_{\omega_{\gamma\ve j}(t)}\nu_\gamma+\Delta_{\omega_{\gamma\ve j}(t)}\dot{\varphi}_{\gamma\ve j}(t)\right)-n+\tr_{\omega_{\gamma\ve j}(t)}\omega_j.
\end{aligned}
\end{equation}
Then
\begin{equation}\label{0321009}
\dl H_{\gamma\ve j}=-\tr_{\omega_{\gamma\ve j}(t)}\omega_j\leqslant0.
\end{equation}
By the maximum principle, we have $H_{\gamma \ve j}(t)\leqslant \sup\limits_XH_{\gamma \ve j}(0)=0$. Therefore,
\begin{equation}\label{0321010}
\dot{\varphi}_{\gamma\ve j}(t)\leqslant \frac{\varphi_{\gamma\ve j}(t)-\varphi_{j}}{t}+n
\end{equation}
for all $j\geqslant1$, $\ve\in(0,1)$ and $t\in(0,T]$. Letting $j\to\infty$ and then $\ve\to0$, we obtain
\begin{equation}\label{0321011}
\dot{\varphi}_{\gamma}(t)\leqslant \frac{\varphi_{\gamma}(t)-\varphi_{0}}{t}+n\ \ \ \text{in}\ \ \ X\setminus D.
\end{equation}
Inequality \eqref{0321003} follows from \eqref{0321011} and Proposition \ref{0320029}.
\end{proof}
\begin{cor}\label{0405001}
For any $0<T<T^\eta_{\max}$, there exists a uniform constant $C$ depending only on $n$, $\omega$, $\theta$, $\eta$ and $T$ such that for any $t\in[0,T]$ and $0<\gamma\ll1$,
\begin{equation}\label{0405002}
t\left(\dot\varphi_\gamma(t)-\psi_\gamma\right)\leqslant\osc\limits_X\varphi_0+Ct\ \ \ \text{in}\ \ \ X\setminus D.
\end{equation}
\end{cor}
\begin{cor}\label{0322011}
For any $0<\delta<T<T^\eta_{\max}$, there exists a uniform constant $C$ depending only on $\|\varphi_0\|_{L^\infty(X)}$, $n$, $\omega$, $\theta$, $\eta$, $\delta$ and $T$ such that for any $t\in[\delta,T]$ and $0<\gamma\ll1$,
\begin{equation}\label{0322012}
\dot\varphi_\gamma(t)\leqslant\psi_\gamma+C\ \ \ \text{in}\ \ \ X\setminus D.
\end{equation}
\end{cor}
Then by using Propositions \ref{0320029} and \ref{0321001}, we give some new arguments to prove a uniform lower bound on the potential in a short time.
\begin{pro}\label{0321017}
There exists a uniform constant $C$ depending only on $n$, $\omega$, $\theta$, $\eta$ and $T$ such that
\begin{equation}\label{0321018}
\varphi_\gamma(t)\geqslant \inf\limits_X\varphi_0+t\psi_\gamma-C\ \ \ \text{on}\ \ \ X
\end{equation}
for all $t\in[0,1]$ and $0<\gamma\ll1$.
\end{pro}
\begin{proof}
We only consider the case $t\in(0,1]$. By using Proposition \ref{0321001}, we have 
\begin{equation}\label{0321019}
\begin{aligned}
\left(\omega_{\gamma t}+\dd\varphi_\gamma(t)\right)^n&=e^{\dot\varphi_\gamma(t)-h}\frac{\omega^n}{|s|_h^{2(1-\gamma)}}\\
&\leqslant e^{\frac{\varphi_\gamma(t)-\varphi_0}{t}+C-(1-\gamma)\log|s|_h^2+\log\frac{\omega^n}{\omega_\gamma^n}}\omega_\gamma^n\\
&\leqslant e^{\frac{\varphi_\gamma(t)-t\psi_\gamma}{t}+\psi_\gamma-\frac{1}{t}\inf\limits_X\varphi_0+C-(1-\gamma)\log|s|_h^2+\log\frac{\omega^n}{\omega_\gamma^n}}\omega_\gamma^n\\
&\leqslant e^{\frac{\varphi_\gamma(t)-t\psi_\gamma}{t}+\frac{1}{t}\left(C-\inf\limits_X\varphi_0\right)}\omega_\gamma^n,
\end{aligned}
\end{equation}
where we use inequality \eqref{03200241} in the last inequality.

For $0<\epsilon\ll1$, we denote 
\begin{equation}\label{0321020}
\phi_{\gamma}(t)=\varphi_\gamma(t)-t\psi_\gamma\ \ \ \text{and}\ \ \ \chi_{\gamma\epsilon}(t)=\phi_\gamma(t)-\epsilon\log|s|_h^2.
\end{equation}
For any $0<t_0\leqslant1$, we let $x_0$ be the minimum point of $\chi_{\gamma\epsilon}(t_0)$. Since $\phi_\gamma(t_0)$ is continuous on $X$, $x_0$ must be attained in $X\setminus D$. At $(t_0,x_0)$, by the maximum principle and \eqref{0320027}, there holds
\begin{equation}\label{0321021}
\begin{split}
\left(\omega_{\gamma t_0}+\dd\varphi_{\gamma}(t_0)\right)^n&=\left(\omega_{\gamma t_0}+t_0\dd\psi_\gamma+\dd\phi_{\gamma}(t_0)\right)^n\\
&=\left(\omega_{\gamma t_0}+t_0\dd\psi_\gamma-\epsilon\theta+\dd\chi_{\gamma\epsilon}(t_0)\right)^n\\
&\geqslant\left(\omega_{\gamma t_0}+t_0\dd\psi_\gamma-\epsilon\theta\right)^n\\
&=\left(t_0\left(\omega_{\gamma t_0}+\dd\psi_\gamma\right)+\left(1-t_0\right)\omega_{\gamma t_0}-\epsilon\theta\right)^n\\
&\geqslant\left(\frac{t_0}{2}\right)^n\left(\omega_{\gamma t_0}+\dd\psi_\gamma\right)^n\\
&\geqslant\frac{t_0^n}{C}\omega_\gamma^n
\end{split}
\end{equation}
for all $0<\gamma,\epsilon\ll1$. Therefore, by using \eqref{0321019}, at $(t_0,x_0)$, we have
\begin{equation}\label{0321022}
\frac{1}{t_0}\phi_\gamma(t_0,x_0)+\frac{1}{t_0}\left(C-\inf\limits_X\varphi_0\right)\geqslant n\log t_0-C,
\end{equation}
that is, 
\begin{equation}\label{0321023}
\phi_\gamma(t_0,x_0)\geqslant \inf\limits_X\varphi_0+nt_0\log t_0-C\geqslant\inf\limits_X\varphi_0-C.
\end{equation}
Hence 
\begin{equation}\label{0321024}
\chi_{\gamma\epsilon}(t_0)\geqslant\phi_\gamma(t_0,x_0)\geqslant\inf\limits_X\varphi_0-C,
\end{equation}
and then
\begin{equation}\label{0321025}
\varphi_\gamma(t_0)\geqslant t_0\psi_\gamma+\epsilon\log|s|_h^2+\inf\limits_X\varphi_0-C.
\end{equation}
Let $\epsilon\to0$, there holds
\begin{equation}\label{0321026}
\varphi_\gamma(t_0)\geqslant \inf\limits_X\varphi_0+t_0\psi_\gamma-C \ \ \ \text{on}\ \ \ X
\end{equation}
for some uniform constant $C$ depending only on $n$, $\omega$, $\theta$, $\eta$ and $T$. Since $t_0$ is arbitrary, we complete the proof.
\end{proof}
\begin{rem}\label{0515010}
In fact, by using similar arguments as in the proof of Proposition \ref{0320029}, we can also give a uniform low bound as follows. There exists a uniform constant $C$ depending only on $n$, $\omega$, $\theta$, $\eta$ and $T$ such that
\begin{equation}\label{0515011}
\varphi_\gamma(t)\geqslant \inf\limits_X\varphi_0+t\psi_\gamma+n(t\log t-t)-Ct\ \ \ \text{on}\ \ \ X
\end{equation}
for all $t\in[0,1]$ and $0<\gamma\ll1$. 

In Proposition \ref{0329010}, we will show a better uniform lower bound which improves the terms $\inf\limits_X\varphi_0$ in \eqref{0321018} and \eqref{0515011} to be $\varphi_0$. It is very important in showing the limit behavior of the limit flow near $t=0$.
\end{rem}
On the foundation of above uniform estimates, we prove a uniform lower bound on $\dot\varphi_\gamma(t)$ in a uniform short time. In this proof, the originality of the auxiliary function lies in the coefficient $t+\sigma$ of the term $\left(\dot\varphi_\gamma(t)-\psi_\gamma\right)$ for some uniform small positive constant $\sigma$, which is only $t$ in previous literature, such as \cite{JSGT, VGAZ2, JWLXZ1, JWLXZ19, LSZ} etc. The main reason is that in this proof, we need to derive a term $\sigma\tr_{\omega_\gamma(t)}\omega_\gamma$ with uniform coefficient $\sigma>0$ for all $t>0$ and $0<\gamma\ll1$ to control the term $\log\frac{\omega^n_{\gamma}(t)}{\omega_\gamma^n}$ coming from $\dot\varphi_\gamma(t)$ by the mean value inequality. But the background metric considered here is $\omega_{\gamma t}+t\dd\psi_\gamma$, due to the coefficient of $\dd\psi_\gamma$ is $t$, we can only derive the term $t\tr_{\omega_\gamma(t)}\omega_\gamma$, which can not imply a uniform lower bound for all $t>0$. So we choose the coefficient $t+\sigma$ in the new auxiliary function.
\begin{pro}\label{0321027}
There exist uniform positive constants $\sigma$ and $\kappa$ depending only on $\eta$, $\theta$ and $\omega$, and  a uniform constant $C$ depending only on $n$, $\omega$, $\theta$, $\eta$ and $T$ such that
\begin{equation}\label{0321028}
\sigma\left(\dot\varphi_\gamma(t)-\psi_\gamma\right)\geqslant-2\osc\limits_X\varphi_0+n\log t-C\ \ \ \text{in}\ \ \ X\setminus D
\end{equation}
for all $t\in(0,\kappa]$ and $0<\gamma\ll1$.
\end{pro}
\begin{proof}
For any $0<\sigma\ll1$ and $\epsilon>0$, we define function
\begin{equation}\label{0321029}
G_\gamma(t)=\left(t+\sigma\right)\left(\dot\varphi_\gamma(t)-\psi_\gamma\right)+\left(\varphi_\gamma(t)-t\psi_\gamma\right)-n\log t-\epsilon\log|s|_h^2.
\end{equation}
For any fix $\gamma$, from the author results (Proposition $3.9$ in \cite{JWLXZ4}, see also Lemma $3.7$ in \cite{LSZ}), $(t+\sigma)\dot\varphi_\gamma(t)$ is bounded from below by $n\sigma\log t$ near $t=0$. And since $\varphi_\gamma(t)$ and $\psi_\gamma$ are continuous on $X$, and $n\sigma\log t-n\log t$ tends to $+\infty$ as $t\to0$, $G_\gamma(t)$ tends to $+\infty$ as $t\to0$ or $x\to D$ and so the minimum point of $G_\gamma(t)$ must be away from both $t=0$ and the divisor $D$, where $G_\gamma(t)$ is smooth.

In $X\setminus D$, direct computations show that
\begin{align}\label{0321030}\nonumber
\left(\frac{\partial}{\partial t}-\Delta_{\omega_{\gamma}(t)}\right)G_\gamma(t)&=\dot\varphi_\gamma(t)-\psi_\gamma+(t+\sigma)\left(\ddot\varphi_\gamma(t)-\Delta_{\omega_{\gamma}(t)}\dot\varphi_\gamma(t)\right)+(t+\sigma)\Delta_{\omega_{\gamma}(t)}\psi_\gamma\\\nonumber
&\ \ \ +\dot\varphi_\gamma(t)-\Delta_{\omega_{\gamma}(t)}\varphi_\gamma(t)-\psi_\gamma+t\Delta_{\omega_{\gamma}(t)}\psi_\gamma-\frac{n}{t}-\epsilon\tr_{\omega_\gamma(t)}\theta\\\nonumber
&=2\dot\varphi_\gamma(t)-2\psi_\gamma+(t+\sigma)\tr_{\omega_\gamma(t)}\nu_\gamma-\Delta_{\omega_{\gamma}(t)}\varphi_\gamma(t)-\frac{n}{t}-\epsilon\tr_{\omega_\gamma(t)}\theta\\
&\ \ \ +\left(2t+\sigma\right)\Delta_{\omega_{\gamma}(t)}\psi_\gamma\\\nonumber
&=2\left(\dot\varphi_\gamma(t)-\psi_\gamma\right)-n+\tr_{\omega_\gamma(t)}\left(\omega+(2t+\sigma)\nu_\gamma+\left(2t+\sigma\right)\dd\psi_\gamma-\epsilon\theta\right)-\frac{n}{t}\\\nonumber
&=\tr_{\omega_\gamma(t)}\left(\sigma\left(\omega+\dd\psi_\gamma\right)+\left((1-\sigma)\omega+(2t+\sigma)\nu_\gamma+2t\dd\psi_\gamma\right)-\epsilon\theta\right)\\\nonumber
&\ \ \ +2\left(\dot\varphi_\gamma(t)-\psi_\gamma\right)-n-\frac{n}{t}.
\end{align}
By using \eqref{03200241}, we have
\begin{equation}\label{0321031}
\begin{split}
\dot\varphi_\gamma(t)-\psi_\gamma&=\log\frac{\left(\omega_{\gamma t}+\dd\varphi_{\gamma}(t)\right)^n}{\omega^n}+h+(1-\gamma)\log|s|_h^2-\psi_\gamma\\
&=\log\frac{\left(\omega_{\gamma t}+\dd\varphi_{\gamma}(t)\right)^n}{\omega_\gamma^n}+\log\frac{\omega_\gamma^n}{\omega^n}+h+(1-\gamma)\log|s|_h^2-\psi_\gamma\\
&\geqslant\log\frac{\left(\omega_{\gamma t}+\dd\varphi_{\gamma}(t)\right)^n}{\omega_\gamma^n}-C.
\end{split}
\end{equation}
On the other hand, if $\nu_\gamma\geqslant-C\omega$ and $\theta\leqslant C\omega$ with some uniform large constant $C>2$, by using \eqref{0320028}, for $0<\sigma<\frac{1}{2(C+1)}$, $0<t<\frac{1}{8C}:=\kappa<\frac{1}{2}$ and $0<\epsilon<\frac{C-2}{4C(C-1)}$, we have
\begin{equation}\label{032103266}
\begin{split}
&\ \ \ (1-\sigma)\omega+(2t+\sigma)\nu_\gamma+2t\dd\psi_\gamma-\epsilon\theta\\
&\geqslant\left(1-\sigma-(2t+\sigma)C-\frac{1}{4\left(C-1\right)}-C\epsilon\right)\omega\\
&\geqslant0.
\end{split}
\end{equation}
Therefore, there exists a uniform constant $C$ depending only on $n$, $\omega$, $\theta$, $\eta$ and $T$ such that
\begin{equation}\label{0321030}
\begin{split}
\left(\frac{\partial}{\partial t}-\Delta_{\omega_{\gamma}(t)}\right)G_\gamma(t)&\geqslant2\log\frac{\left(\omega_{\gamma t}+\dd\varphi_{\gamma}(t)\right)^n}{\omega_\gamma^n}+\sigma\tr_{\omega_\gamma(t)}\omega_\gamma-\frac{C}{t}\\
&\geqslant2\log\frac{\left(\omega_{\gamma t}+\dd\varphi_{\gamma}(t)\right)^n}{\omega_\gamma^n}+\sigma n\left(\frac{\omega_\gamma^n}{\omega^n_{\gamma}(t)}\right)^{\frac{1}{n}}-\frac{C}{t},
\end{split}
\end{equation}
where we use the mean value inequality in the last inequality. 

Consider function 
\begin{equation}\label{0321031}
f(s)=\frac{n\sigma}{2}s^{\frac{1}{n}}-2\log s.
\end{equation}
There exists a constant $S$ depending only on $\sigma$, $l$ and $n$ such that $f(s)\geqslant0$ when $s\geqslant S$ or $0<s\leqslant1$.

Assume that $(t_0,x_0)$ is the minimum point of $G_\gamma(t)$ in $(0,\kappa]\times (X\setminus D)$. If $\frac{\omega_\gamma^n}{\omega^n_{\gamma}(t)}\geqslant S$ at $(t_0,x_0)$, then
\begin{equation}\label{0321032}
\left(\frac{\partial}{\partial t}-\Delta_{\omega_{\gamma}(t)}\right)G_\gamma(t)\geqslant\frac{n\sigma}{2}\left(\frac{\omega_\gamma^n}{\omega^n_{\gamma}(t)}\right)^{\frac{1}{n}}-\frac{C}{t}.
\end{equation}
The maximum principle implies that at $(t_0,x_0)$,
\begin{equation}\label{0321033}
\left(\frac{\omega_\gamma^n}{\omega^n_{\gamma}(t)}\right)\leqslant \left(\frac{2C}{\sigma nt_0}\right)^{n}
\end{equation}
and then
\begin{equation}\label{0321034}
\log\frac{\omega^n_{\gamma}(t)}{\omega_\gamma^n}\geqslant n\log t_0-C.
\end{equation}

If $\frac{\omega_\gamma^n}{\omega^n_{\gamma}(t)}<S$  at $(t_0,x_0)$, then
\begin{equation}\label{0321035}
\log\frac{\omega^n_{\gamma}(t)}{\omega_\gamma^n}\geqslant -\log S.
\end{equation} 
Combining above arguments with \eqref{03200241} and Proposition \ref{0321017}, we have
\begin{equation}\label{0321036}
\begin{split}
&\ \ \ \left(t+\sigma\right)\left(\dot\varphi_\gamma(t)-\psi_\gamma\right)+\left(\varphi_\gamma(t)-t\psi_\gamma\right)-n\log t-\epsilon\log|s|_h^2\\
&\geqslant\left(t_0+\sigma\right)\left(\dot\varphi_\gamma(t_0,x_0)-\psi_\gamma(x_0)\right)+\left(\varphi_\gamma(t_0,x_0)-t_0\psi_\gamma(x_0)\right)-n\log t_0\\
&\geqslant\left(t_0+\sigma\right)\left(\log\frac{\left(\omega_{\gamma t}+\dd\varphi_{\gamma}(t)\right)^n}{\omega^n}+h+(1-\gamma)\log|s|_h^2-\psi_\gamma\right)(t_0,x_0)\\
&\ \ \ +\inf\limits_X\varphi_0-C-n\log t_0\\
&=\left(t_0+\sigma\right)\left(\log\frac{\left(\omega_{\gamma t}+\dd\varphi_{\gamma}(t)\right)^n}{\omega_\gamma^n}+\log\frac{\omega_\gamma^n}{\omega^n}+h+(1-\gamma)\log|s|_h^2-\psi_\gamma\right)(t_0,x_0)\\
&\ \ \ +\inf\limits_X\varphi_0-C-n\log t_0\\
&\geqslant(t_0+\sigma)(n\log t_0-C)+\inf\limits_X\varphi_0-C-n\log t_0\\
&=-(1-t_0-\sigma)n\log t_0+\inf\limits_X\varphi_0-C\\
&\geqslant\inf\limits_X\varphi_0-C,
\end{split}
\end{equation}
where we use $t_0+\sigma<1$ uniformly. After letting $\epsilon\to0$ and using Proposition \ref{0320029}, we obtain
\begin{equation}\label{0321037}
\begin{split}
\left(t+\sigma\right)\left(\dot\varphi_\gamma(t)-\psi_\gamma\right)\geqslant-\osc\limits_X\varphi_0+n\log t-C.
\end{split}
\end{equation}
Then by using Corollary \ref{0405001}, we complete the proof .
\end{proof}
\begin{rem}\label{0516001}
In fact, the constant $C$ in \eqref{0321028} also depends on $\sigma$. Since $\sigma$ depends only on $\eta$, $\theta$ and $\omega$, we will no longer mention that the constant depends on it after we fix one $\sigma$ in the following arguments.
\end{rem}
Next, we extend above uniform estimates in short time to all time in $(0,T]$.
\begin{pro}\label{0322100}
For any $T\in\left(\kappa,T^\eta_{\max}\right)$, where $\kappa$ comes from Proposition \ref{0321027}, there exists a uniform constant $C$ depending only on $\|\varphi_0\|_{L^\infty(X)}$, $n$, $\omega$, $\theta$, $\eta$ and $T$ such that
\begin{equation}\label{032210666}
(T-t)\left(\dot\varphi_{\gamma}(t)-\psi_\gamma\right)\geqslant-C\ \ \ \text{in}\ \ \ X\setminus D
\end{equation}
for all $t\in[\kappa,T]$ and $0<\gamma\ll1$.

Furthermore, for any $T'\in(\kappa,T)$, there exists a uniform constant $C$ depending only on $\|\varphi_0\|_{L^\infty(X)}$, $n$, $\omega$, $\theta$, $\eta$, $T'$ and $T$ such that
\begin{equation}\label{0322101}
\dot\varphi_\gamma(t)\geqslant\psi_\gamma-C\ \ \ \text{in}\ \ \ X\setminus D
\end{equation}
for all $t\in[\kappa,T']$ and $0<\gamma\ll1$.
\end{pro}
\begin{proof}
Define 
\begin{equation}\label{0322102}
K_{\gamma\ve j}(t)=(T-t)\left(\dot\varphi_{\gamma\ve j}(t)-\psi_\gamma\right)+\left(\varphi_{\gamma\ve j}(t)-t\psi_\gamma\right)+nt-\epsilon\log|s|_h^2.
\end{equation}
Direct computations show that in $X\setminus D$,
\begin{equation}\label{0322103}
\begin{split}
\left(\frac{\partial}{\partial t}-\Delta_{\omega_{\gamma\ve j}(t)}\right)K_{\gamma\ve j}(t)&=(T-t)\tr_{\omega_{\gamma\ve j}(t)}\nu_\gamma+\tr_{\omega_{\gamma\ve j}(t)}(\omega_{\gamma t}-\epsilon\theta)+T\Delta_{\omega_{\gamma\ve j}(t)}\psi_\gamma\\
&=\tr_{\omega_{\gamma\ve j}(t)}\left(\omega+T\nu_\gamma+T\dd\psi_\gamma-\epsilon\theta\right).
\end{split}
\end{equation}
By using \eqref{0320016} and \eqref{0320028}, we have
\begin{equation}\label{0322104}
\omega+T\nu_\gamma+T\dd\psi_\gamma-\epsilon\theta\geqslant\left(\frac{1}{C}-\frac{1}{8\left(C-1\right)}-C\epsilon\right)\omega\geqslant0
\end{equation}
if we choose $\epsilon$ sufficiently small. Hence
\begin{equation}\label{0322105}
\left(\frac{\partial}{\partial t}-\Delta_{\omega_{\gamma\ve j}(t)}\right)K_{\gamma\ve j}(t)\geqslant0.
\end{equation}
Since $K_{\gamma\ve j}(t)$ tends to $+\infty$ as $x\to D$, the minimum point of $K_{\gamma\ve j}(t)$ on $[t_0,T]\times X$ with $t_0\geqslant0$ must be outside $D$. The maximum principle shows that
\begin{equation}\label{0322104}
\begin{split}
&\ \ \ (T-t)\left(\dot\varphi_{\gamma\ve j}(t)-\psi_\gamma\right)+\left(\varphi_{\gamma\ve j}(t)-t\psi_\gamma\right)+nt-\epsilon\log|s|_h^2\\
&\geqslant\inf\limits_X\left((T-t_0)\left(\dot\varphi_{\gamma\ve j}(t_0)-\psi_\gamma\right)+\left(\varphi_{\gamma\ve j}(t_0)-t_0\psi_\gamma\right)\right).
\end{split}
\end{equation}
Let $\epsilon\to0$, and then $j\to+\infty$ and $\ve\to0$, there holds
\begin{equation}\label{0322105}
\begin{split}
&\ \ \ (T-t)\left(\dot\varphi_{\gamma}(t)-\psi_\gamma\right)+\left(\varphi_{\gamma}(t)-t\psi_\gamma\right)+nt\\
&\geqslant\inf\limits_X\left((T-t_0)\left(\dot\varphi_{\gamma}(t_0)-\psi_\gamma\right)+\left(\varphi_{\gamma}(t_0)-t_0\psi_\gamma\right)\right).
\end{split}
\end{equation}
If we choose $t_0=\kappa$, by using Propositions \ref{0320029}, \ref{0321017} and \ref{0321027}, we have
\begin{equation}\label{0322106}
(T-t)\left(\dot\varphi_{\gamma}(t)-\psi_\gamma\right)\geqslant-C\ \ \ \text{in}\ \ \ X\setminus D
\end{equation}
for all $t\in[\kappa,T]$ and $0<\gamma\ll1$, where constant $C$ depends only on $\|\varphi_0\|_{L^\infty(X)}$, $n$, $\omega$, $\theta$, $\eta$ and $T$. Then it implies for any $T'\in(\kappa,T)$, there exists a uniform constant $C$ depending only on $\|\varphi_0\|_{L^\infty(X)}$, $n$, $\omega$, $\theta$, $\eta$, $T'$ and $T$ such that
\begin{equation}\label{0322107}
\dot\varphi_\gamma(t)\geqslant\psi_\gamma-C\ \ \ \text{in}\ \ \ X\setminus D
\end{equation}
for all $t\in[\kappa,T']$ and $0<\gamma\ll1$.
\end{proof}
\begin{rem}\label{0516002}
In fact, the constant $C$ in \eqref{032210666} also depends on $\kappa$. Since $\kappa$ depends only on $\eta$, $\theta$ and $\omega$, we will no longer mention that the constant depends on it in the following arguments.
\end{rem}
Combining Proposition \ref{0321027} with Proposition \ref{0322100}, we have the following uniform estimates.
\begin{cor}\label{0405003}
There exist uniform constants $\sigma$, and  a uniform constant $C$ depending only on $\|\varphi_0\|_{L^\infty(X)}$, $n$, $\omega$, $\theta$, $\eta$ and $T$ such that
\begin{equation}\label{0405004}
\sigma\left(\dot\varphi_\gamma(t)-\psi_\gamma\right)\geqslant n\log t-C\ \ \ \text{in}\ \ \ X\setminus D
\end{equation}
for all $t\in(0,T]$ and $0<\gamma\ll1$.
\end{cor}
\begin{pro}\label{0322013}
For any $0<T<T^\eta_{\max}$, there exists a uniform constant $C$ depending only on $\|\varphi_0\|_{L^\infty(X)}$, $n$, $\omega$, $\theta$, $\eta$ and $T$ such that for any $t\in[0,T]$ and $0<\gamma\ll1$,
\begin{equation}\label{0322014}
\varphi_\gamma(t)\geqslant t\psi_\gamma-C\ \ \ \text{on}\ \ \ X.
\end{equation}
\end{pro}
\begin{proof}
From Proposition \ref{0321017}, we need only consider the case $t\geqslant1$. In this case, Corollary \ref{0405003} implies that there exists a uniform constant $C$ depending only on $\|\varphi_0\|_{L^\infty(X)}$, $n$, $\omega$, $\theta$, $\eta$ and $T$ such that
\begin{equation}\label{0405004666}
\dot\varphi_\gamma(t)-\psi_\gamma\geqslant -C\ \ \ \text{in}\ \ \ X\setminus D
\end{equation}
for all $t\in[1,T]$ and $0<\gamma\ll1$.

Integrating \eqref{0405004666} from $1$ to $t$ on both sides, we have
\begin{equation}\label{0322015}
\varphi_\gamma(t)-\varphi_\gamma(1)\geqslant(t-1)\psi_\gamma-C(t-1)\ \ \ \text{on}\ \ \ X,
\end{equation}
which implies that 
\begin{equation}\label{0322016}
\varphi_\gamma(t)\geqslant t\psi_\gamma+\varphi_\gamma(1)-\psi_\gamma-C(t-1)\ \ \ \text{on}\ \ \ X.
\end{equation}
Then inequality \eqref{0322014} follows from Proposition \ref{0321017}.
\end{proof}

\section{High order uniform estimates}\label{Bou2}
In this section, we show the high order uniform estimates along the twisted conical K\"ahler-Ricci flow. The most important one is the uniform Laplacian $C^2$-estimate, where the key step is to compare $\omega_\gamma(t)$ with the conical K\"ahler metric $\omega_\gamma$. Since we only obtain the uniform estimates on the metric potential with respect to the metric $\omega_{\gamma t}+t\dd\psi_\gamma$ which is a conical K\"ahler metric depending on $t$, there is on uniform lower bound on $\omega_{\gamma t}+t\dd\psi_\gamma$ for all $t>0$ and so we need some new arguments. First, we recall the following uniform estimates on the holomorphic bisectional curvature of $\omega_\gamma$, which was proved by Guenancia \cite{G11}.
\begin{lem}[Theorem $3.2$ in \cite{G11}]\label{032400466}
There exists a constant $C$ depending only on $X$ such that for all $\gamma\in\left(0,\frac{1}{2}\right]$, the holomorphic bisectional curvature of $\omega_\gamma$ is bounded by $C$.
\end{lem}
Next, we prove the uniform Laplacian $C^2$-estimate by constructing a new auxiliary function whose originality lies in the new terms $(t-\tau)\left(\dot\varphi_\gamma(t)-\psi_\gamma\right)$ (with uniform small positive constant $\tau$) and $\log\left(t-\delta\right)$ added. The main reason for adding the first term is that in this proof, we need to derive a negative term $-\tau\tr_{\omega_\gamma(t)}\omega_\gamma$ with uniform coefficient $\tau>0$ for all $0<t<T$ and $0<\gamma\ll1$ to cancel the term $\tr_{\omega_{\gamma}(t)}\omega_{\gamma}$ coming from the evolution equation of $\log\tr_{\omega_{\gamma}(t)}\omega_{\gamma}$. But the background metric considered here is $\omega_{\gamma t}+t\dd\psi_\gamma$, which can only contribute the term $-t\tr_{\omega_\gamma(t)}\omega_\gamma$ and so there is no uniform bound for all $t>0$. So we add the term $A(t-\tau)\left(\dot\varphi_\gamma(t)-\psi_\gamma\right)$ in the new auxiliary function. Moreover, the second new term $\log\left(t-\delta\right)$ ensures the maximum point of the auxiliary function can not be attained at the time $t=\delta$.

\begin{pro}\label{0229011}
For any $0<T<T^\eta_{\max}$, there exists a uniform constant $C$ depending only on $\|\varphi_0\|_{L^\infty(X)}$, $n$, $\omega$, $\theta$, $\eta$ and $T$ such that 
\begin{equation}\label{0229012}
\log\tr_{\omega_{\gamma}(t)}\omega_{\gamma}\leqslant \frac{C}{t^2}.
\end{equation}
for all $t\in\left(0,T\right]$ and $0<\gamma\ll1$.
\end{pro}
\begin{proof}
For any $0<\delta<T$ and $0<\epsilon\ll1$, we define 
\begin{equation}\label{0323001}
L_\gamma(t)=\left(t-\delta\right)\log\tr_{\omega_{\gamma}(t)}\omega_{\gamma}-A\left(\varphi_\gamma(t)-t\psi_\gamma\right)+A(t-\tau)\left(\dot\varphi_{\gamma}(t)-\psi_\gamma\right)+An\log\left(t-\delta\right)+\epsilon\log|s|_h^2,
\end{equation}
where $A$ is a large positive uniform constant to be determined later, $\tau$ is a fixed uniform positive constant satisfying 
\begin{equation}\label{0323002}
(1-\tau)\omega+\tau\nu_\gamma\geqslant(1-\tau-C\tau)\omega\geqslant\frac{\omega}{2}\ \ \ \text{and}\ \ \ \tau\leqslant \sigma
\end{equation}
for all $0<\gamma\ll1$, and $\sigma$ is the constant in Corollary \ref{0405003}. 

First, by Chern-Lu inequality \cite{SSC, YCL} (see also Proposition $7.1$ in \cite{JMR}), in $M\setminus D$, we have
\begin{equation}\label{2018060801}
\Delta_{\omega_\gamma(t)}\log \tr_{\omega_{\gamma}(t)}\omega_{\gamma}\geqslant\frac{\left(Ric(\omega_\gamma(t)),\omega_\gamma\right)_{\omega_\gamma(t)}}{\tr_{\omega_{\gamma}(t)}\omega_{\gamma}}-C \tr_{\omega_{\gamma}(t)}\omega_{\gamma},
\end{equation}
where $(\ ,\ )_{\omega_\gamma(t)}$ is the inner product with respect to $\omega_\gamma(t)$ and the constant $C$ depends on the upper bound on the holomorphic bisectional curvature of $\omega_\gamma$ which is uniformly bounded by Lemma \ref{032400466}. In $M\setminus D$, we also have
\begin{equation}\label{2018060802}\frac{\partial}{\partial t}\log \tr_{\omega_{\gamma}(t)}\omega_{\gamma}=\frac{\left(Ric(\omega_\gamma(t))-\eta,\omega_\gamma\right)_{\omega_\gamma(t)}}{\tr_{\omega_{\gamma}(t)}\omega_{\gamma}}.
\end{equation}
By using $(\ref{2018060801})$ and $(\ref{2018060802})$, we have
\begin{equation}\label{0324004}
\left(\frac{\partial}{\partial t}-\Delta_{\omega_\gamma(t)}\right)\log \tr_{\omega_{\gamma}(t)}\omega_{\gamma}\leqslant C \tr_{\omega_{\gamma}(t)}\omega_{\gamma}+C,
\end{equation}
where constant $C$ independent of $\gamma$. Then direct computations combining with Corollary \ref{0405003} show that
\begin{equation}\label{0324005}
\begin{aligned}
\left(\frac{\partial}{\partial t}-\Delta_{\omega_\gamma(t)}\right) L_{\gamma}(t)&\leqslant\log\tr_{\omega_{\gamma}(t)}\omega_{\gamma}+C\tr_{\omega_{\gamma}(t)}\omega_{\gamma}-A\left(\dot\varphi_\gamma(t)-\psi_\gamma\right)+A\Delta_{\omega_\gamma(t)}\varphi_\gamma(t)+\frac{An}{t-\delta}\\
&\ \ \ -At\Delta_{\omega_\gamma(t)}\psi_\gamma+A\left(\dot\varphi_\gamma(t)-\psi_\gamma\right)+A(t-\tau)\tr_{\omega_{\gamma}(t)}\nu_\gamma+\epsilon\tr_{\omega_\gamma(t)}\theta\\
&\ \ \ +A(t-\tau)\Delta_{\omega_\gamma(t)}\psi_\gamma+C\\
&\leqslant\log\tr_{\omega_{\gamma}(t)}\omega_{\gamma}+C\tr_{\omega_{\gamma}(t)}\omega_{\gamma}+An-A\tr_{\omega_\gamma(t)}\omega_{\gamma t}-A\tau\Delta_{\omega_\gamma(t)}\psi_\gamma+\frac{An}{t-\delta}\\
&\ \ \ +A(t-\tau)\tr_{\omega_{\gamma}(t)}\nu_\gamma+\epsilon\tr_{\omega_\gamma(t)}\theta+C\\
&\leqslant\log\tr_{\omega_{\gamma}(t)}\omega_{\gamma}+C\tr_{\omega_{\gamma}(t)}\omega_{\gamma}+An+C+\epsilon\tr_{\omega_\gamma(t)}\theta+\frac{An}{t-\delta}\\
&\ \ \ -A\tr_{\omega_\gamma(t)}\left(\omega+\tau\nu_\gamma+\tau\dd\psi_\gamma\right).
\end{aligned}
\end{equation}
We first deal with the last term in above inequality. By \eqref{0323002}, there holds
\begin{equation}\label{0324006}
\omega+\tau\nu_\gamma+\tau\dd\psi_\gamma=(1-\tau)\omega+\tau\nu_\gamma+\tau\left(\omega+\dd\psi_\gamma\right)\geqslant \frac{\omega}{2}+\tau\omega_\gamma.
\end{equation}
Hence 
\begin{equation}\label{0324007}
\begin{aligned}
\left(\frac{\partial}{\partial t}-\Delta_{\omega_\gamma(t)}\right) L_{\gamma}(t)&\leqslant\log\tr_{\omega_{\gamma}(t)}\omega_{\gamma}+C\tr_{\omega_{\gamma}(t)}\omega_{\gamma}-A\tau\tr_{\omega_\gamma(t)}\omega_\gamma+\frac{C}{t-\delta}\\
&\ \ \ +\tr_{\omega_\gamma(t)}\left(\epsilon\theta-\frac{A}{2}\omega\right)\\
&\leqslant\log\tr_{\omega_{\gamma}(t)}\omega_{\gamma}-2\tr_{\omega_{\gamma}(t)}\omega_{\gamma}+\frac{C}{t-\delta}\\
&\leqslant-\tr_{\omega_{\gamma}(t)}\omega_{\gamma}+\frac{C}{t-\delta},
\end{aligned}
\end{equation}
where we choose $A$ large enough in the second inequality, and we use $-x+\log x\leqslant0$ on $(0,+\infty)$ in the last inequality.

Let $(t_0,x_0)$ be the maximum point of $L_\gamma(t)$ on $[\delta, T]\times \left(X\setminus D\right)$. From the authors' result ($(3.58)$ in \cite{JWLXZ4}), there exists a constant $C_\gamma$ such that for any $t\in\left[\delta,T\right]$, 
\begin{equation}\label{0324003}
\frac{1}{C_\gamma}\omega_\gamma\leqslant\omega_{\gamma}(t)\leqslant C_\gamma\omega_\gamma.
\end{equation}
Combining this with Proposition \ref{0322013}, Corollaries \ref{0405001} and \ref{0405003}, we have
\begin{equation}\label{0509001}
\begin{split}
L_\gamma(t)&\leqslant T\log\left(C_\gamma n\right)+An\log\left(t-\delta\right)-\frac{A\tau }{\sigma}n\log t+C+\epsilon\log|s|_h^2\\
&\leqslant T\log\left(C_\gamma n\right)+An\log\left(t-\delta\right)-\frac{A\tau }{\sigma}n\log\delta+C+\epsilon\log|s|_h^2.
\end{split}
\end{equation}
We conclude that $L_\gamma(t)$ tends to $-\infty$ as $x\to D$ or $x\to\delta$, hence the maximum point $(t_0,x_0)$ of $L_\gamma(t)$ must be in $(\delta,T]\times (X\setminus D)$. By the maximum principle, at $(t_0,x_0)$, we have
\begin{equation}\label{0324009}
\begin{split}
\tr_{\omega_{\gamma}(t)}\omega_{\gamma}\leqslant \frac{C}{t-\delta}+C,
\end{split}
\end{equation}
where $C$ is a constant depending only on $\|\varphi_0\|_{L^\infty(X)}$, $n$, $\omega$, $\theta$, $\eta$ and $T$. Then by using Propositions \ref{0320029} and \ref{0322013}, and Corollaries \ref{0405001} and \ref{0405003} again, since $\tau\leqslant\sigma$, we have
\begin{align}\label{0405005}\nonumber
&\ \ \ \left(t-\delta\right)\log\tr_{\omega_{\gamma}(t)}\omega_{\gamma}-\frac{C}{t-\delta}+An\log\left(t-\delta\right)+\epsilon\log|s|_h^2-C\\\nonumber
&\leqslant\left(t-\delta\right)\log\tr_{\omega_{\gamma}(t)}\omega_{\gamma}-\frac{C}{t}+An\log\left(t-\delta\right)+\epsilon\log|s|_h^2-C\\\nonumber
&\leqslant\left(t-\delta\right)\log\tr_{\omega_{\gamma}(t)}\omega_{\gamma}-A\left(\varphi_\gamma(t)-t\psi_\gamma\right)+A(t-\tau)\left(\dot\varphi_{\gamma}(t)-\psi_\gamma\right)+An\log\left(t-\delta\right)+\epsilon\log|s|_h^2\\
&\leqslant\left(t_0-\delta\right)\log\frac{C}{t_0-\delta}-\frac{A\tau}{\sigma}n\log t_0+An\log\left(t_0-\delta\right)+\epsilon\log|s|_h^2(x_0)+C\\\nonumber
&\leqslant-\left(t_0-\delta\right)\log(t_0-\delta)+\left(1-\frac{\tau}{\sigma}\right)An\log\left(t_0-\delta\right)+\epsilon\log|s|_h^2(x_0)+C\\\nonumber
&\leqslant\left(1-\frac{\tau}{\sigma}\right)An\log T+C,
\end{align}
and so
\begin{equation}\label{0324010}
\left(t-\delta\right)\log\tr_{\omega_{\gamma}(t)}\omega_{\gamma}\leqslant \frac{C}{t-\delta}+C-\epsilon\log|s|_h^2\ \ \ \text{in}\ \ \ X\setminus D,
\end{equation}
where $C$ is a constant depending only on $\|\varphi_0\|_{L^\infty(X)}$, $n$, $\omega$, $\theta$, $\eta$ and $T$. 
After letting $\epsilon\to0$, we have
\begin{equation}\label{0405006}
\log\tr_{\omega_{\gamma}(t)}\omega_{\gamma}\leqslant \frac{C}{(t-\delta)^2}\ \ \ \text{in}\ \ \ X\setminus D.
\end{equation}
Then we complete the proof after we let $\delta\to0$.
\end{proof}
\begin{rem}\label{0509005}
The constant $\tau$ in the definition of $L_\gamma(t)$ in \eqref{0323001} can be chosen to be equal to $\sigma$ in Corollary \ref{0405003}. In fact, by the choosing of $\sigma$ in \eqref{032103266}, we can also choose $\sigma$ satisfying \eqref{0323002}. 
\end{rem}
\begin{cor}\label{0325001}
For any $0<T<T^\eta_{\max}$, there exists a uniform constant $C$ depending only on $\|\varphi_0\|_{L^\infty(X)}$, $n$, $\omega$, $\theta$, $\eta$ and $T$ such that 
\begin{equation}\label{0325002}
e^{-\frac{C}{t^2}}\omega_\gamma\leqslant\omega_{\gamma}(t)\leqslant e^{\frac{C}{t^2}}\omega_{\gamma}\ \ \ \text{in}\ \ \ X\setminus D
\end{equation}
for all $t\in\left(0,T\right]$ and $0<\gamma\ll1$.
\end{cor}
\begin{proof}
From Proposition \ref{0229011}, there exists a uniform constant $C$ depending only on $\|\varphi_0\|_{L^\infty(X)}$, $n$, $\omega$, $\theta$, $\eta$ and $T$ such that for any $t\in\left(0,T\right]$ and $0<l\ll1$,
\begin{equation}\label{0325003}
\tr_{\omega_{\gamma}(t)}\omega_{\gamma}\leqslant e^{\frac{C}{t^2}}\ \ \ \text{in}\ \ \ X\setminus D.
\end{equation}
Since
\begin{equation}\label{0325004}
\tr_{\omega_\gamma}\omega_{\gamma}(t)\leqslant n\left(\frac{\omega_{\gamma}^n(t)}{\omega_\gamma^n}\right)\left(\tr_{\omega_\gamma(t)}\omega_\gamma\right)^{n-1}
\end{equation}
and 
\begin{equation}\label{0325005}
\begin{split}
\frac{\omega_{\gamma}^n(t)}{\omega_\gamma^n}&=\exp\left(\dot\varphi_\gamma(t)-\log\frac{\omega_\gamma^n}{\omega^n}-h-(1-\gamma)\log|s|_h^2\right)\\
&=\exp\left(\dot\varphi_\gamma(t)-\psi_\gamma-\log\frac{\omega_\gamma^n}{\omega^n}-h-(1-\gamma)\log|s|_h^2+\psi_\gamma\right)\\
&\leqslant e^{\frac{C}{t}},
\end{split}
\end{equation}
where we use \eqref{03200241} and Corollary \ref{0405001} in the last inequality, there exists a uniform constant $C$ depending only on $\|\varphi_0\|_{L^\infty(X)}$, $n$, $\omega$, $\theta$, $\eta$ and $T$ such that for any $t\in\left(0,T\right]$ and $0<l\ll1$,
 \begin{equation}\label{0325006}
\tr_{\omega_\gamma}\omega_{\gamma}(t)\leqslant e^{\frac{C}{t^2}}\ \ \ \text{in}\ \ \ X\setminus D.
\end{equation}
Combining \eqref{0325003} and \eqref{0325006}, we obtain \eqref{0325002}.
\end{proof}

Therefore, for any $0<\delta<T<T^\eta_{\max}$ and $K\subset\subset X\setminus D$, there exists a uniform constant $C$ depending only on $\|\varphi_0\|_{L^\infty(X)}$, $n$, $\omega$, $\theta$, $\eta$, $\delta$, $T$ and $dist_{\omega}(K,D)$ such that 
\begin{equation}\label{0324011}
\omega\leqslant\omega_{\gamma}(t)\leqslant\omega\ \ \ \text{on}\ \ \ [\delta, T]\times K.
\end{equation}
Then Evans-Krylov-Safonov's estimates (see \cite{NVK}) imply the high order uniform estimates.
\begin{pro}\label{217} For any $0<\delta<T<T^\eta_{\max}$, $k\in\mathbb{N}^{+}$ and $K\subset\subset X\setminus D$, there exist uniform constants $C_k$ depending only on $\|\varphi_0\|_{L^\infty(X)}$, $k$, $n$, $\omega$, $\theta$, $\eta$, $\delta$, $T$ and $dist_{\omega}(K,D)$, such that for any $0<\gamma\ll1$,
\begin{equation}
\|\varphi_{\gamma}(t)\|_{C^{k}\left([\delta,T]\times K\right)}\leqslant C_k.
\end{equation}
\end{pro}

\section{Existence}\label{Exi}

In this section, we show the limit behavior flow of the conical K\"ahler-Ricci flow \eqref{CK} as its cone angle tends to zero. We prove the following convergence theorem.
\begin{thm}\label{0506001}
As the cone angle tends to $0$, there exists a subsequence ${\color{OrangeRed} (CKRF^\eta_{\gamma_i})}$ of the twisted conical K\"ahler-Ricci flow converging locally in the smooth sense outside the divisor and globally in the sense of currents to a K\"ahler-Ricci flow
\begin{equation}
\left\{
\begin{aligned}
 &\ \frac{\partial \omega(t)}{\partial t}=-{\rm Ric}(\omega(t))+[D]+\eta\\
 &\ \ \ \ \ \ \ \ \ \ \ \ \ \ \ \ \ \ \ \ \ \ \ \ \ \ \ \ \ \ \ \ \ \ \ \ \ \ \ \ , \\
 &\ \omega(t)|_{t=0}=\hat\omega
\end{aligned}
\right.\tag{$CKRF^{\eta}_{0}$}
\end{equation}
which exists on $[0,T^\eta_{\max})\times X$ in the following sense. 
\begin{itemize}
  \item For any $0<T<T^\eta_{\max}$, there exists a constant $C$ such that for any $t\in(0,T)$,
\begin{equation}\label{0509006}
e^{-\frac{C}{t^2}}\omega_o\leqslant\omega(t)\leqslant e^{\frac{C}{t^2}}\omega_o\ \ \ \text{in}\ \ \ X\setminus D;
\end{equation}
   \item  In $(0,T^\eta_{\max})\times(X\setminus D)$, $\omega(t)$ satisfies the smooth twisted K\"ahler-Ricci flow
 \begin{equation}\label{0331008}
\frac{\partial \omega(t)}{\partial t}=-{\rm Ric}(\omega(t))+\eta;
\end{equation}
  \item On $(0,T^\eta_{\max})\times X$, $\omega(t)$ satisfies K\"ahler-Ricci flow \eqref{PK} in the sense of currents;
  \item There exists a metric potential $\varphi(t)\in\mathcal{C}^{\infty}\left((0,T^\eta_{\max})\times(X\setminus D)\right)$ with respect to $\omega_t$ such that $\omega(t)=\omega_{t}+\dd\varphi(t)$,
  \begin{equation}\label{03310089}
  \lim\limits_{t\rightarrow0^{+}}\|\varphi(t)-\varphi_0\|_{L^{1}(X)}=0\ \ \ \text{\rm and}\ \ \ \lim\limits_{t\rightarrow0^{+}}\varphi(t)=\varphi_0\ \ \ \text{in}\ \ \ X\setminus D;
\end{equation}
\item On $[0,T^\eta_{\max})$, $\|\varphi(t)-t\psi_o\|_{L^\infty(X)}\leqslant C$;
\item For any $0<T<T^\eta_{\max}$, there exist constants $a$ and $C$ such that for any $t\in(0,T)$,
\begin{equation}\label{0406001}
-C+a\log t\leqslant\dot\varphi(t)-\psi_o\leqslant\frac{C}{t}.
\end{equation}
\end{itemize}
\end{thm}
\begin{proof}
For any $[\delta, T]\times K\subset\subset (0,T^\eta_{\max})\times (X\setminus D)$, it is easy to see that $[\delta, T]\times K\subset\subset (0,T^\eta_{\gamma,\max})\times (X\setminus D)$ for $0<\gamma\ll1$. Then for any $k\in\mathbb{N}^{+}$, $\|\varphi_{\gamma}(t)\|_{C^{k}([\delta,T]\times K)}$ is uniformly bounded according to Proposition \ref{217}. Let $\delta$ approximate to $0$, $T$ approximate to $T^\eta_{\max}$ and $K$ approximate to $X\setminus D$, by diagonal rule, we get a sequence $\{\gamma_i\}$, such that $\varphi_{\gamma_i}(t)$ converges in $C^\infty_{loc}$-topology in $ (0,T^\eta_{\max})\times (X\setminus D)$ to a function $\varphi(t)$ which is smooth in $(0,T^\eta_{\max})\times (X\setminus D)$ and satisfies equation
\begin{equation}\label{223}  
\frac{\partial \varphi(t)}{\partial t}=\log\frac{(\omega_{t}+\sqrt{-1}\partial\bar{\partial}\varphi(t))^{n}}{\omega^{n}}+h+\log|s|_{h}^{2}.
\end{equation}
Combining these arguments with Corollary \ref{0325001}, for any $0<T<T^\eta_{\max}$, there exists a uniform constant $C$ depending only on $\|\varphi_0\|_{L^\infty(X)}$, $n$, $\omega$, $\theta$, $\eta$ and $T$ such that for any $t\in(0,T]$,
\begin{equation}\label{221}
e^{-\frac{C}{t^2}}\omega_o\leqslant\omega(t)=\omega_t+\dd\varphi(t)\leqslant e^{\frac{C}{t^2}}\omega_o\ \ \text{in}\ X\setminus D.
\end{equation}
Hence we prove the first two conditions, and the last two conditions follow from Propositions \ref{0320029} and \ref{0322013}, and Corollaries \ref{0405001} and \ref{0405003}.

Next, we prove that $\omega(t)$ satisfies equation \eqref{PK} in the sense of currents on $(0,T^\eta_{\max})\times X$. Let $\zeta:=\zeta(t,x)$ be a smooth $(n-1,n-1)$-form with compact support in $(0,T^\eta_{\max})\times X$. Without loss of generality, we assume that its compact support is included in $(\delta,T)\times X$ ($0<\delta<T<T^\eta_{\max}$). 

By using the definitions \eqref{0320018}, \eqref{0328001} and the Dominated convergence theorem, we have
\begin{equation}\label{0328002}
\int_{X}\dd\psi_\gamma\wedge\zeta\xrightarrow{\gamma\rightarrow 0}\int_{X}\dd\psi_o\wedge\zeta.
\end{equation}
On $[\delta,T]\times (X\setminus D)$, since $\log\frac{\omega_{\gamma}^{n}(t)}{\omega^{n}}+(1-\gamma)\log|s|_h^2-\psi_\gamma$, $\log\frac{\omega^{n}(t)}{\omega^{n}}+\log|s|_h^2-\psi_o$, $\varphi_{\gamma}(t)-t\psi_\gamma$ and $\varphi(t)-t\psi_o$ are uniformly bounded, by using the Dominated convergence theorem and integration by parts, we have
\begin{equation}\label{201503201}
\begin{split}
&\ \int_{X}\frac{\partial\omega_{\gamma}(t)}{\partial t}\wedge\zeta=\int_{X}\left(\nu_\gamma+\sqrt{-1}\partial\bar{\partial}\frac{\partial\varphi_{\gamma}(t)}{\partial t}\right)\wedge\zeta\\
&=\int_{X}\dd\left(\frac{\partial\varphi_{\gamma}(t)}{\partial t}-\psi_\gamma\right)\wedge\zeta+\int_{X}\dd\psi_\gamma\wedge\zeta+\int_X\nu_\gamma\wedge\zeta\\
&=\int_{X}\left(\frac{\partial\varphi_{\gamma}(t)}{\partial t}-\psi_\gamma\right)\dd\zeta+\int_{X}\dd\psi_\gamma\wedge\zeta+\int_X\nu_\gamma\wedge\zeta\\
&=\int_{X}\left(\log\frac{\omega_{\gamma}^{n}(t)}{\omega^{n}}+(1-\gamma)\log|s|_h^2+h-\psi_\gamma\right)\sqrt{-1}\partial\bar{\partial}\zeta+\int_{X}\dd\psi_\gamma\wedge\zeta+\int_X\nu_\gamma\wedge\zeta\\
&\xrightarrow{\gamma\rightarrow 0}\int_{X}\left(\log\frac{\omega^{n}(t)}{\omega^{n}}+\log|s|_h^2+h-\psi_o\right)\sqrt{-1}\partial\bar{\partial}\zeta+\int_{X}\dd\psi_o\wedge\zeta+\int_X\nu\wedge\zeta\\
&=\int_{X}\dd\left(\log\frac{\omega^{n}(t)}{\omega^{n}}+\log|s|_h^2+h-\psi_o\right)\wedge\zeta+\int_{X}\dd\psi_o\wedge\zeta+\int_X\nu\wedge\zeta\\
&=\int_{X} (-{\rm Ric} (\omega(t))+[D]+\eta)\wedge\zeta.
\end{split}
\end{equation}
At the same time, there also holds
\begin{equation}\label{201503202}
\begin{split}
\int_X \omega_{\gamma}(t)\wedge\frac{\partial\zeta}{\partial t}&=\int_X \left(\omega+t\nu_\gamma\right)\wedge\frac{\partial\zeta}{\partial t}+\int_X \dd\varphi_{\gamma}(t)\wedge\frac{\partial\zeta}{\partial t}\\
&=\int_X \left(\omega+t\nu_\gamma\right)\wedge\frac{\partial\zeta}{\partial t}+\int_X \dd\left(\varphi_{\gamma}(t)-t\psi_\gamma\right)\wedge\frac{\partial\zeta}{\partial t}+t\int_X \dd\psi_\gamma\wedge\frac{\partial\zeta}{\partial t}\\
&=\int_X \left(\omega+t\nu_\gamma\right)\wedge\frac{\partial\zeta}{\partial t}+\int_X\left(\varphi_{\gamma}(t)-t\psi_\gamma\right)\dd\frac{\partial\zeta}{\partial t}+t\int_X \dd\psi_\gamma\wedge\frac{\partial\zeta}{\partial t}\\
&\xrightarrow{\gamma\rightarrow 0}\int_X\left(\omega+t\nu\right)\wedge\frac{\partial\zeta}{\partial t}+\int_X\left(\varphi(t)-t\psi_o\right)\dd\frac{\partial\zeta}{\partial t}+t\int_X \dd\psi_o\wedge\frac{\partial\zeta}{\partial t}\\
&=\int_X\left(\omega+t\nu\right)\wedge\frac{\partial\zeta}{\partial t}+\int_X\dd\left(\varphi(t)-t\psi_o\right)\wedge\frac{\partial\zeta}{\partial t}+t\int_X \dd\psi_o\wedge\frac{\partial\zeta}{\partial t}\\
&=\int_X \omega(t)\wedge\frac{\partial\zeta}{\partial t}.
\end{split}
\end{equation}
On the other hand,
\begin{align}\label{201503203}\nonumber
\frac{\partial}{\partial t}\int_X \omega_{\gamma}(t)\wedge\zeta
&=\int_X\omega_\gamma(t)\wedge\frac{\partial\zeta}{\partial t}+\int_X\left(\nu_\gamma+\dd\frac{\partial\varphi_{\gamma}(t)}{\partial t}\right)\wedge\zeta\\\nonumber
&=\int_X\omega_{\gamma t}\wedge\frac{\partial\zeta}{\partial t}+\int_X \dd\left(\varphi_{\gamma}(t)-t\psi_\gamma\right)\wedge\frac{\partial\zeta}{\partial t}+t\int_X\dd\psi_\gamma\wedge\frac{\partial\zeta}{\partial t}\\\nonumber
&\ \ \ +\int_X\nu_\gamma\wedge\zeta+\int_X\dd\left(\frac{\partial\varphi_{\gamma}(t)}{\partial t}-\psi_\gamma\right)\wedge\zeta+\int_X\dd\psi_\gamma\wedge\zeta\\\nonumber
&=\int_X\omega_{\gamma t}\wedge\frac{\partial\zeta}{\partial t}+\int_X \left(\varphi_{\gamma}(t)-t\psi_\gamma\right)\dd\frac{\partial\zeta}{\partial t}+t\int_X\dd\psi_\gamma\wedge\frac{\partial\zeta}{\partial t}\\ \nonumber
&\ \ \ +\int_X\nu_\gamma\wedge\zeta+\int_X\left(\frac{\partial\varphi_{\gamma}(t)}{\partial t}-\psi_\gamma\right)\dd\zeta+\int_X\dd\psi_\gamma\wedge\zeta\\ 
&\xrightarrow{\gamma\rightarrow 0}\int_X\omega_{t}\wedge\frac{\partial\zeta}{\partial t}+\int_X \left(\varphi(t)-t\psi_o\right)\dd\frac{\partial\zeta}{\partial t}+t\int_X\dd\psi_o\wedge\frac{\partial\zeta}{\partial t}\\\nonumber
&\ \ \ +\int_X\nu\wedge\zeta+\int_X\left(\frac{\partial\varphi(t)}{\partial t}-\psi_o\right)\dd\zeta+\int_X\dd\psi_o\wedge\zeta\\\nonumber
&=\int_X\omega_{t}\wedge\frac{\partial\zeta}{\partial t}+\int_X \dd\left(\varphi(t)-t\psi_o\right)\wedge\frac{\partial\zeta}{\partial t}+t\int_X\dd\psi_o\wedge\frac{\partial\zeta}{\partial t}\\\nonumber
&\ \ \ +\int_X\nu\wedge\zeta+\int_X\dd\left(\frac{\partial\varphi(t)}{\partial t}-\psi_o\right)\wedge\zeta+\int_X\dd\psi_o\wedge\zeta\\\nonumber
&=\int_X\omega(t)\wedge\frac{\partial\zeta}{\partial t}+\int_X \frac{\partial\omega(t)}{\partial t}\wedge\frac{\partial\zeta}{\partial t}\\\nonumber
&=\frac{\partial}{\partial t}\int_X \omega(t)\wedge\zeta.
\end{align}
Combining equality
\begin{equation}\label{0328003}
\frac{\partial}{\partial t}\int_X \omega_{\gamma}(t)\wedge\zeta=\int_{X}\frac{\partial\omega_{\gamma}(t)}{\partial t}\wedge\zeta+\int_X \omega_{\gamma}(t)\wedge\frac{\partial\zeta}{\partial t}
\end{equation}
with equalities $(\ref{201503201})$-$(\ref{201503203})$, on $[\delta,T]$, we have
\begin{equation}\label{201503204}
\frac{\partial}{\partial t}\int_X \omega(t)\wedge\zeta=\int_{X} \left(-{\rm Ric} (\omega(t))+[D]+\eta\right)\wedge\zeta+\int_X \omega(t)\wedge\frac{\partial\zeta}{\partial t}.
\end{equation}
Since ${\rm Supp}\zeta\subset(\delta,T)\times X$, we have 
\begin{equation}\label{2015032040909}
\int_0^{T^\eta_{\max}}\frac{d}{d t}\int_X \omega(t)\wedge\zeta\ dt=0.
\end{equation}
Then integrating form $0$ to $\infty$ on both sides of equation $(\ref{201503204})$,
\begin{equation*}
\begin{split}
\int_{(0,T^\eta_{\max})\times X} \frac{\partial \omega(t)}{\partial t}\wedge\zeta~dt
&=-\int_{(0,T^\eta_{\max})\times X} \omega(t)\wedge\frac{\partial\zeta}{\partial t}~dt=-\int_{0}^{T^\eta_{\max}}\int_X \omega(t)\wedge\frac{\partial\zeta}{\partial t}~dt\\
&=\int_{0}^{T^\eta_{\max}}\int_{X} \left(-{\rm Ric} (\omega(t))+ [D]+\eta\right)\wedge\zeta~dt\\
&=\int_{(0,T^\eta_{\max})\times X} \left(-{\rm Ric} (\omega(t))+[D]+\eta\right)\wedge\zeta~dt.
\end{split}
\end{equation*}
By the arbitrariness of $\zeta$, we prove that $\omega(t)$ satisfies K\"ahler-Ricci flow \eqref{PK} in the sense of currents on $(0,T^\eta_{\max})\times X$. We complete the proof of the third condition.

At last, we prove $\varphi(t)$ we just obtained converges to $\varphi_0$ globally in $L^1$-sense and point-wise in $X\setminus D$ as $t\to0$. In \cite{JWLXZ4}, the authors and Zhang prove that  
\begin{equation}\label{0516005}
\lim\limits_{t\rightarrow0^{+}}\|\varphi_\gamma(t)-\varphi_{0}\|_{L^{1}(X)}=0\ \ \ \text{and}\ \ \ \lim\limits_{t\rightarrow0^{+}}\varphi_\gamma(t)=\varphi_{0}\ \ \text{in}\ \ X\setminus D.
\end{equation}
But all these limits depend on $\gamma$, so we can not let $\gamma$ tends to $0$ directly. We need some new arguments to show the limit behavior of $\varphi(t)$ near $t=0$. More precisely, our method is to construct super-solution and sub-solution and then to study the limit behavior of these solutions near $t=0$. We first give an upper bound on $\varphi(t)$. Fix $\ve>0$, we consider the smooth parabolic Monge-Amp\`ere equation
\begin{equation}\label{SC}
\left\{
\begin{aligned}
 &\ \frac{\partial \varphi_{\ve j}(t)}{\partial t}=\log\frac{\left(\omega_{t}+\dd\varphi_{\ve j}(t)\right)^n}{\omega^n}+h+\log\left(\ve^2+|s|_h^2\right)\\
 &\ \ \ \ \ \ \ \ \ \ \ \ \ \ \ \ \ \ \ \ \ \ \ \ \ \ \ \ \ \ \ \ \ \ \ \ \ \ \ \ \ \ \ \ \ \ \ \ \ \ \ \ \ \ \ \ \ \ \ \ \ \ \ \ \ \ \ \ \ \ \ \ \ \ \ \ \ ,\\
 &\ \varphi_{\ve j}(t)|_{t=0}=\varphi_{j}
\end{aligned}
\right.
\end{equation}
where $\varphi_j$ is the smooth strictly $\omega$-psh function decreasing to $\varphi_0$ as $j\nearrow\infty$ in \eqref{0220002}. From Tsuji \cite{Tsuji} and Tian-Zhang's \cite{Tianzzhang} results, there exists a unique smooth solution $\varphi_{\ve j}(t)$ on $[0,T]\times X$ with $0<T<T^\eta_{\max}$ for all $j$.
\begin{pro}\label{0329001}
In $(0,T]\times (X\setminus D)$, we have
\begin{equation}\label{0329002}
\varphi_\gamma(t)+t\gamma\log|s|_h^2\leqslant\varphi_{\ve j}(t)
\end{equation}
for all $0<\gamma\ll1$.
\end{pro}
\begin{proof}
For any $0<\delta<T$, we consider the function
\begin{equation}\label{0329003}
Z(t)=\varphi_\gamma(t)+t\gamma\log|s|_h^2-\varphi_{\ve j}(t)\ \ \ \text{on}\ \ \ [\delta,T]\times X.
\end{equation}
Direct computations show that, in $X\setminus D$,
\begin{equation}\label{0329004}
\begin{split}
\frac{\partial}{\partial t}Z(t)&=\log\frac{\left(\omega_t+\dd\left(\varphi_\gamma(t)+t\gamma\log|s|_h^2\right)\right)^n}{\left(\omega_t+\dd\varphi_{\ve j}(t)\right)^n}+\log|s|_h^2-\log\left(\ve^2+|s|_h^2\right)\\
&\leqslant\log\frac{\left(\omega_t+\dd\left(\varphi_\gamma(t)+t\gamma\log|s|_h^2\right)\right)^n}{\left(\omega_t+\dd\varphi_{\ve j}(t)\right)^n}\\
&=\log\frac{\left(\omega_t+\dd\varphi_{\ve j}(t)+\dd Z(t)\right)^n}{\left(\omega_t+\dd\varphi_{\ve j}(t)\right)^n}.
\end{split}
\end{equation}
Let $(t_0,x_0)$ be the maximum point of $Z(t)$ on $[\delta,T]\times X$. Since $Z(t)$ tends to $-\infty$ as $x\to D$, $x_0$ must be in $X\setminus D$. Then by the maximum principle $t_0=\delta$, which implies that
\begin{equation}\label{0329005}
\varphi_\gamma(t)+t\gamma\log|s|_h^2-\varphi_{\ve j}(t)\leqslant\sup\limits_X\left(\varphi_\gamma(\delta)+\delta\gamma\log|s|_h^2-\varphi_{\ve j}(\delta)\right)\leqslant\sup\limits_X\left(\varphi_\gamma(\delta)-\varphi_{\ve j}(\delta)\right).
\end{equation}
It follow from Hartogs' Lemma that, as $\delta\to0$,
\begin{equation}\label{0329006}
\sup\limits_X\left(\varphi_\gamma(\delta)-\varphi_{\ve j}(\delta)\right)\to\sup\limits_X\left(\varphi_0-\varphi_{j}\right)\leqslant0,
\end{equation}
we obtain \eqref{0329002} after letting  $\delta\to0$ in \eqref{0329005}.
\end{proof}
We obtain an upper bound for $\varphi(t)$ if we let $\gamma\to0$ in \eqref{0329002}.
\begin{cor}\label{0329007}
In $(0,T]\times (X\setminus D)$, we have
\begin{equation}\label{0329008}
\varphi(t)\leqslant\varphi_{\ve j}(t).
\end{equation}
\end{cor}
For the lower bound on $\varphi(t)$, we construct a sub-solution to derive it.  
\begin{pro}\label{0329010}
There exists a uniform constant $l\geqslant1$ depending only on $\eta$, $\theta$ and $\omega$, and a uniform constant $C$ depending only on $\|\varphi_0\|_{L^\infty(X)}$, $n$, $\omega$, $\theta$, $\eta$ and $T$ such that 
\begin{equation}\label{0329011}
\varphi_0+t\psi_\gamma+n(t\log t-t)-Ct\leqslant\varphi_\gamma(t)\ \ \ \text{in}\ \ \ X\setminus D
\end{equation}
for all $t\in(0,\frac{1}{2l})$ and $0<\gamma\ll1$.
\end{pro}
\begin{proof}
Fix $l\geqslant1$ such that
\begin{equation}\label{0329012}
\frac{1}{2l-1}< T\ \ \ \text{and}\ \ \ (2l-1)\omega+\nu_\gamma\geqslant\frac{\omega}{2}
\end{equation}
for all $0<\gamma\ll1$, and fix $0<t_0\ll1$ such that
\begin{equation}\label{03290131}
(2l-1)\omega+(1+2lt_0)\nu_\gamma\geqslant0
\end{equation}
and 
\begin{equation}\label{0329014}
\omega+\left(1+2lt_0\right)\dd\psi_\gamma=\omega_\gamma+2lt_0\dd\psi_\gamma\geqslant\frac{1}{2}\omega_\gamma.
\end{equation}

We consider equation
\begin{equation}\label{0229032}
(\omega+\dd u_{\gamma t_0})^n=e^{u_{\gamma t_0}-h-2l\varphi_\gamma(t_0)}\frac{\omega^n}{\vert s\vert_h^{2(1-\gamma)}}.
\end{equation}
It follows from Ko{\l}odziej's results \cite{K000} that there exists a unique continuous solution $u_{\gamma t_0}$ to above equation. Now we write this equation as 
\begin{equation}\label{0329009}
\begin{split}
&\ \ \ \left(\omega+\left(1+2lt_0\right)\dd\psi_\gamma+\dd \left(u_{\gamma t_0}-\left(1+2lt_0\right)\psi_\gamma\right)\right)^n\\
&=e^{\left(u_{\gamma t_0}-\left(1+2lt_0\right)\psi_\gamma\right)-h-2l\left(\varphi_\gamma(t_0)-t_0\psi_\gamma\right)+\psi_\gamma+\log\frac{\omega^n}{\omega_\gamma^n}-(1-\gamma)\log|s|_h^2}\omega_\gamma^n.
\end{split}
\end{equation}

We first prove that there exists a uniform constant $C$ depending only on $\|\varphi_0\|_{L^\infty(X)}$, $n$, $\omega$, $\theta$, $\eta$ and $T$ such that 
\begin{equation}\label{0330001}
\|u_{\gamma t_0}-\left(1+2lt_0\right)\psi_\gamma\|_{L^\infty(X)}\leqslant C.
\end{equation}
For any $\epsilon>0$, we consider function 
\begin{equation}\label{0330002}
F(x)=u_{\gamma t_0}-\left(1+2lt_0\right)\psi_\gamma-\epsilon\log|s|_h^2
\end{equation}
whose minimum point $x_0$ must be in $X\setminus D$ since $F(x)$ tends to $+\infty$ as $x\to D$. At point $x=x_0$, by the maximum principle and \eqref{0329014} we have
\begin{equation}\label{0330003}
\begin{split}
&\ \ \ \left(\omega+\left(1+2lt_0\right)\dd\psi_\gamma+\dd \left(u_{\gamma t_0}-\left(1+2lt_0\right)\psi_\gamma\right)\right)^n\\
&=\left(\omega+\left(1+2lt_0\right)\dd\psi_\gamma-\epsilon\theta+\dd F(x_0)\right)^n\\
&\geqslant\left(\omega+\left(1+2lt_0\right)\dd\psi_\gamma-\epsilon\theta\right)^n\\
&\geqslant\left(\frac{1}{2}\omega_\gamma-\epsilon\theta\right)^n\geqslant \frac{1}{4^n}\omega_\gamma^n,
\end{split}
\end{equation}
and by \eqref{03200241} and Proposition \ref{0321017},
\begin{equation}\label{0330004}
\begin{split}
&\ \ \ e^{\left(u_{\gamma t_0}-\left(1+2lt_0\right)\psi_\gamma\right)-h-2l\left(\varphi_\gamma(t_0)-t_0\psi_\gamma\right)+\psi_\gamma+\log\frac{\omega^n}{\omega_\gamma^n}-(1-\gamma)\log|s|_h^2}\omega_\gamma^n\\
&\leqslant e^{\left(u_{\gamma t_0}-\left(1+2lt_0\right)\psi_\gamma\right)+C}\omega_\gamma^n.
\end{split}
\end{equation}
Hence there exists a uniform constant $C$ depending only on $\|\varphi_0\|_{L^\infty(X)}$, $n$, $\omega$, $\theta$, $\eta$ and $T$ such that
\begin{equation}\label{0330005}
u_{\gamma t_0}(x_0)-\left(1+2lt_0\right)\psi_\gamma(x_0)\geqslant C\ \ \ \text{on}\ \ \ X,
\end{equation}
which implies
\begin{equation}\label{0330006}
u_{\gamma t_0}-\left(1+2lt_0\right)\psi_\gamma-\epsilon\log|s|_h^2\geqslant C\ \ \ \text{on}\ \ \ X,
\end{equation}
and
\begin{equation}\label{0330007}
u_{\gamma t_0}-\left(1+2lt_0\right)\psi_\gamma\geqslant C\ \ \ \text{on}\ \ \ X
\end{equation}
 after we let $\epsilon\to0$. By similar arguments we can also obtain an upper bound on $u_{\gamma t_0}-\left(1+2lt_0\right)\psi_\gamma$.

Next, we set 
\begin{equation}\label{0330008}
M(t)=(1-2lt)\varphi_\gamma(t_0)+tu_{\gamma t_0}+n(t\log t-t).
\end{equation}
Since $\varphi_\gamma(t_0)$ and $u_{\gamma t_0}$ are continuous, 
\begin{equation}\label{0330009}
\lim\limits_{t\to0^+}\|M(t)-\varphi_\gamma(t_0)\|_{L^\infty(X)}=0.
\end{equation}
For $0\leqslant t\leqslant\frac{1}{2l}-t_0$, by using \eqref{03290131}, we have 
\begin{equation}\label{0330010}
\begin{split}
\omega_{\gamma(t+t_0)}+\dd M(t)&=\omega+(t+t_0)\nu_\gamma+(1-2lt)\dd\varphi_\gamma(t_0)+t\dd u_{\gamma  t_0}\\
&=(1-2lt)\omega_{\varphi_\gamma(t_0)}+t\omega_{u_{\gamma  t_0}}+t\left((2l-1)\omega+(1+2lt_0)\nu_\gamma\right)\\
&\geqslant t\omega_{u_{\gamma  t_0}}.
\end{split}
\end{equation}
Hence
\begin{equation}\label{0330011}
\left(\omega_{\gamma(t+t_0)}+\dd M(t)\right)^n\geqslant t^n\omega_{u_{\gamma  t_0}}^n=t^ne^{u_{\gamma t_0}-h-2l\varphi_\gamma(t_0)}\frac{\omega^n}{\vert s\vert_h^{2(1-\gamma)}},
\end{equation}
which implies that
\begin{equation}\label{0330012}
\begin{split}
&\ \ \ \ \log\frac{(\omega_{\gamma(t+t_0)}+\dd M(t))^n}{\omega^n}+h+(1-\gamma)\log\vert s\vert_h^2\\
&\geqslant n\log t+u_{\gamma t_0}-2l\varphi_\gamma(t_0)=\frac{\partial M(t)}{\partial t}.
\end{split}
\end{equation}
So $M(t)$ is a sub-solution to equation \eqref{CC} with initial value $\varphi_{\gamma}(t_0)$. Since $\varphi_\gamma(t+t_0)$ is a solution with initial value $\varphi_\gamma(t_0)$, then by the maximum principle Proposition \ref{0313003} we have
\begin{equation}\label{0330013}
M(t)\leqslant\varphi_\gamma(t+t_0),
\end{equation}
and by Proposition \ref{0320029} and \eqref{0330007}
\begin{equation}\label{0330014}
\begin{split}
\varphi_\gamma(t+t_0)&\geqslant(1-2lt)\varphi_\gamma(t_0)+tu_{\gamma t_0}+n(t\log t-t)\\
&=\varphi_\gamma(t_0)-2lt\left(\varphi_\gamma(t_0)-t_0\psi_\gamma\right)+t\left(u_{\gamma t_0}-\left(1+2lt_0\right)\psi_\gamma\right)+n(t\log t-t)+t\psi_\gamma\\
&\geqslant\varphi_\gamma(t_0)+t\psi_\gamma+n(t\log t-t)-Ct.
\end{split}
\end{equation}
By using Theorem \ref{0510001}, \eqref{0507002} and Remark \ref{0502001}, after letting $t_0\to0$, we have
\begin{equation}\label{0330015}
\varphi_0+t\psi_\gamma+n(t\log t-t)-Ct\leqslant \varphi_\gamma(t)\ \ \ \text{in}\ \ \ X\setminus D
\end{equation}
for all $t\in(0,\frac{1}{2l})$ and $0<l\ll1$.
\end{proof}
\begin{cor}\label{0330016}
There exist a uniform constant $l\geqslant1$ depending only on $\eta$, $\theta$ and $\omega$, and a uniform constant $C$ depending only on $\|\varphi_0\|_{L^\infty(X)}$, $n$, $\omega$, $\theta$, $\eta$ and $T$ such that 
\begin{equation}\label{0330017}
\varphi_0+t\psi_o+n\left(t\log t-t\right)-Ct\leqslant\varphi(t)\ \ \ \text{in}\ \ \ X\setminus D
\end{equation}
for all $t\in(0,\frac{1}{2l})$.
\end{cor}

\begin{pro}\label{0229029}
There holds
\begin{equation}\label{0229030}
\lim\limits_{t\rightarrow0^{+}}\|\varphi(t)-\varphi_{0}\|_{L^{1}(X)}=0.
\end{equation}
and
\begin{equation}\label{0331001}
\lim\limits_{t\to0^+}\varphi(t)=\varphi_0\ \ \ \text{in}\ \ \ X\setminus D.
\end{equation}
\end{pro}
\begin{proof}
By using Corollaries \ref{0329007} and \ref{0330016}, we have
\begin{equation}\label{0329009}
t\psi_o+n\left(t\log t-t\right)-Ct\leqslant\varphi(t)-\varphi_0\leqslant\varphi_{\ve j}(t)-\varphi_j+\varphi_{j}-\varphi_0,
\end{equation}
and so
\begin{equation}\label{0331004}
\|\varphi(t)-\varphi_0\|_{L^1(X)}\leqslant\|\varphi_{\ve j}(t)-\varphi_j\|_{L^1(X)}+\|\varphi_{j}-\varphi_0\|_{L^1(X)}+Ct,
\end{equation}
where we use the fact 
\begin{equation}\label{0331005}
\int_X\left|\psi_o\right|dV_\omega\leqslant C
\end{equation}
for some uniform constant $C$ in the last inequality. Since
\begin{equation}\label{032901011}
\|\varphi_{j}-\varphi_0\|_{L^1(X)}\to0\ \ \ \text{as}\ \ \ j\to+\infty\ \ \ \text{and}\ \ \ \|\varphi_{\ve j}(t)-\varphi_j\|_{\mathcal{C}^\infty(X)}\to0\ \ \ \text{as}\ \ \ t\to0,
\end{equation}
for any $\epsilon>0$, there exists $j$ such that
\begin{equation}\label{0329011}
\|\varphi_{j}-\varphi_0\|_{L^1(X)}\leqslant\frac{\epsilon}{2},
\end{equation}
and for such $j$, there exists a $\delta>0$ such that for any $t\in(0,\delta)$,
\begin{equation}\label{0329012}
\|\varphi_{\ve j}(t)-\varphi_j\|_{\mathcal{C}^\infty(X)}\leqslant\frac{\epsilon}{4}\ \ \ \text{and}\ \ \ Ct\leqslant\frac{\epsilon}{4}.
\end{equation}
So for any $\epsilon>0$, there exists a $\delta>0$ such that for any $t\in(0,\delta)$
\begin{equation}\label{03290130}
\|\varphi(t)-\varphi_0\|_{L^1(X)}\leqslant\epsilon,
\end{equation}
which means 
\begin{equation}\label{0329013}
\lim\limits_{t\to0^+}\|\varphi(t)-\varphi_0\|_{L^1(X)}=0.
\end{equation}

On the other hand, since at any point $x\in X\setminus D$, $\varphi_j(x)$ decreases to $\varphi_0(x)$ as $j\to+\infty$, and $t\psi_o(x)$ and $\|\varphi_{\ve j}(t)-\varphi_j\|_{\mathcal{C}^\infty(X)}$ converge to $0$ as $t\to0^+$, we also conclude that 
\begin{equation}\label{0331006}
\lim\limits_{t\to0^+}\varphi(t)=\varphi_0\ \ \ \text{in}\ \ \ X\setminus D.
\end{equation}
We complete the proof of this proposition, which implies the forth condition in Theorem \ref{0506001}.
\end{proof}
Combining all above arguments together, we complete the proof of Theorem \ref{0506001}.
\end{proof}
\begin{defi}\label{0510002}
Assume that $\hat\omega\in [\omega]$ is a positive closed current with potential $\varphi_0\in L^\infty(X)$. By saying $\omega(t)$ is a weak solution to the twisted cusp K\"ahler-Ricci flow
\begin{equation}
\left\{
\begin{aligned}
 &\ \frac{\partial \omega(t)}{\partial t}=-{\rm Ric}(\omega(t))+[D]+\eta\\
 &\ \ \ \ \ \ \ \ \ \ \ \ \ \ \ \ \ \ \ \ \ \ \ \ \ \ \ \ \ \ \ \ \ \ \ \ \ \ \ \  \\
 &\ \omega(t)|_{t=0}=\hat\omega
\end{aligned}
\right.\tag{$CKRF^{\eta}_{0}$}
\end{equation}
on $(0,T')\times X$, we mean that it satisfies the following conditions.
\begin{itemize}
  \item For any $0<T<T'$, there exists a constant $C$ such that for any $t\in(0,T)$,
\begin{equation}\label{050900688}
e^{-\frac{C}{t^2}}\omega_o\leqslant\omega(t)\leqslant e^{\frac{C}{t^2}}\omega_o\ \ \ \text{in}\ \ \ X\setminus D;
\end{equation}
   \item  In $(0,T')\times(X\setminus D)$, $\omega(t)$ satisfies the smooth twisted K\"ahler-Ricci flow
 \begin{equation}\label{033100888}
\frac{\partial \omega(t)}{\partial t}=-{\rm Ric}(\omega(t))+\eta;
\end{equation}
  \item On $(0,T')\times X$, $\omega(t)$ satisfies K\"ahler-Ricci flow \eqref{PK} in the sense of currents;
  \item There exists a metric potential $\varphi(t)\in\mathcal{C}^{\infty}\left((0,T')\times(X\setminus D)\right)$ with respect to $\omega_t$ such that $\omega(t)=\omega_{t}+\dd\varphi(t)$,
  \begin{equation}\label{033100888}
  \lim\limits_{t\rightarrow0^{+}}\|\varphi(t)-\varphi_0\|_{L^{1}(X)}=0\ \ \ \text{\rm and}\ \ \ \lim\limits_{t\rightarrow0^{+}}\varphi(t)=\varphi_0\ \ \ \text{in}\ \ \ X\setminus D;
\end{equation}
\item On $[0,T')$, $\|\varphi(t)-t\psi_o\|_{L^\infty(X)}\leqslant C$;
\item For any $0<T<T'$, there exist constants $a$ and $C$ such that for any $t\in(0,T)$,
\begin{equation}\label{040600188}
-C+a\log t\leqslant\dot\varphi(t)-\psi_o\leqslant\frac{C}{t}.
\end{equation}
\end{itemize}
\end{defi}
\begin{thm}\label{0330018}
There exists a weak solution $\omega(t)=\omega+\dd\varphi(t)$ to the twisted cusp K\"ahler-Ricci flow \eqref{PK} on $(0,T^\eta_{\max})\times X$, where $\varphi(t)$ is obtained in \eqref{223}.
\end{thm}

\section{Uniqueness}\label{Uni}

In this section, we consider the uniqueness of the cusp K\"ahler-Ricci flow \eqref{PK} in the sense of Definition \ref{0510002}. In order to overcome the problems from the weak initial metric and the singular term in the equation, we apply Di Nezza-Lu's idea in \cite{NL2017} which first considering positive time and then taking a limit to the initial time on smooth equation to the singular equation, and we also use the techniques introduced by the authors in \cite{JWLXZ19}, which ensure the maximum point always attained outside the divisor.

First, we show a lower bound on the solution to equation \eqref{PC} near the initial time $t=0$.
\begin{pro}\label{0401005}
Assume that $\phi(t)$ is a weak solution to equation \eqref{PC} with initial value $\phi_0\in L^\infty(X)$ which is a $\omega$-psf function, that is, $\omega_t+\dd\phi(t)$ is a weak solution to equation \eqref{PK} in the sense of Definition \ref{0510002}. There exist a constant $l\geqslant1$ depending only on $\eta$, $\theta$ and $\omega$, and a constant $C$ depending only on $\|\phi_0\|_{L^\infty(X)}$, $n$, $\omega$, $\theta$, $\eta$ and $T$ such that 
\begin{equation}\label{0401006}
\phi_0+t\psi_o+n(t\log t-t)-Ct\leqslant\phi(t)\ \ \ \text{in}\ \ \ X\setminus D
\end{equation}
for all $t\in(0,\frac{1}{2l})$.
\end{pro}
\begin{proof}
Fix $l\geqslant1$ such that
\begin{equation}\label{0401007}
\frac{1}{2l-1}< T\ \ \ \text{and}\ \ \ (2l-1)\omega+\nu\geqslant\frac{\omega}{2},
\end{equation}
and fix $0<t_0\ll1$ such that
\begin{equation}\label{0401008}
(2l-1-2lt_0)\omega+(1+2lt_0)\nu\geqslant0.
\end{equation}

We consider equation
\begin{equation}\label{0401010}
(\omega+\dd u_{t_0})^n=e^{u_{t_0}-h}\frac{\omega^n}{\vert s\vert_h^{2}}.
\end{equation}
It follows from Kobayashi \cite{Kob84} and Tian-Yau's results \cite{TY87} that there exists a solution $u_{t_0}\in\mathcal{C}^\infty(X\setminus D)$ to above equation. Furthermore, from Guenancia's results \cite{G11}, there exists a uniform constant $C$ independent of $t_0$ such that
\begin{equation}\label{0401011}
\|u_{t_0}-\psi_o\|_{L^\infty(X\setminus D)}\leqslant C.
\end{equation}

We set 
\begin{equation}\label{0402001}
M(t)=(1-2lt)\phi(t_0)+2ltt_0\psi_o+tu_{t_0}+n(t\log t-t)-Ct
\end{equation}
for some uniform constant $C$ independent of $t_0$ and $t$. For $0\leqslant t\leqslant\frac{1}{2l}-t_0$, by using \eqref{0401008}, we have 
\begin{equation}\label{0402002}
\begin{split}
&\ \ \ \omega_{(t+t_0)}+\dd M(t)\\
&=\omega+(t+t_0)\nu+(1-2lt)\dd\phi(t_0)+2ltt_0\dd\psi_o+t\dd u_{t_0}\\
&=(1-2lt)\omega_{\phi(t_0)}+2ltt_0\omega_o+t\omega_{u_{t_0}}+t\left((2l-1-2lt_0)\omega+(1+2lt_0)\nu\right)\\
&\geqslant 2lt_0\omega_{\phi(t_0)}+t\omega_{u_{t_0}}.
\end{split}
\end{equation}
Hence
\begin{equation}\label{0402003}
\left(\omega_{t+t_0}+\dd M(t)\right)^n\geqslant t^n\omega_{u_{t_0}}^n=t^ne^{u_{t_0}-h}\frac{\omega^n}{\vert s\vert_h^{2}},
\end{equation}
which implies that
\begin{equation}\label{0402004}
\begin{split}
&\ \ \ \ \log\frac{(\omega_{t+t_0}+\dd M(t))^n}{\omega^n}+h+\log\vert s\vert_h^2\\
&\geqslant n\log t+u_{t_0}\\
&\geqslant n\log t+u_{t_0}-2l\left(\phi(t_0)-t_0\psi_o\right)-C\\
&=\frac{\partial M(t)}{\partial t}.
\end{split}
\end{equation}

For any $a>0$. We consider $\Phi(t)=M(t)-\phi(t+t_0)+a\log|s|_h^2$ on $[0,\frac{1}{2l}-t_0]\times (X\setminus D)$. Since
\begin{equation}\label{0402005}
\begin{split}
\Phi(t)&=(1-2lt)\phi(t_0)+2ltt_0\psi_o+tu_{t_0}+n(t\log t-t)-Ct-\phi(t+t_0)+a\log|s|_h^2\\
&= (1-2lt)\left(\phi(t_0)-t_0\psi_o\right)+t(u_{t_0}-\psi_o)-\left(\phi(t+t_0)-(t+t_0)\psi_o\right)\\
&\ \ \ +n(t\log t-t)-Ct+a\log|s|_h^2
\end{split}
\end{equation}
and there exists a uniform constant $C$ such that
\begin{equation}\label{0402006}
\|(1-2lt)\left(\phi(t_0)-t_0\psi_o\right)+t(u_{t_0}-\psi_0)-\left(\phi(t+t_0)-(t+t_0)\psi_o\right)\|_{L^\infty(X\setminus D)}\leqslant C,
\end{equation}
$\Phi(t)$ tends to $-\infty$ as $x\to D$.

If we denote
\begin{equation}\label{0402007}
\hat{\Delta}=\int_0^1 g_{s}^{i\bar{j}}\frac{\partial^2}{\partial z^i\partial\bar{z}^j}ds,
\end{equation}
where $g_s$ is the metric corresponding to $\omega_s=\omega_{(t+t_0)}+s\dd M(t)+(1-s)\dd \phi(t+t_0)$. Then $\Phi(t)$ evolves along the following equation
\begin{equation}\label{0402008}
  \frac{\partial \Phi(t)}{\partial t}\leqslant\hat{\Delta}\Phi(t)-a\hat{\Delta}\log|s|_h^{2}.
\end{equation}
Since $\omega_{(t+t_0)}+\dd M(t)$ and $\omega_{\phi(t)}$ are bounded from below by $c_0\omega_o$ and $\omega_o\geqslant c_1\omega$, we obtain
\begin{equation}\label{0402009}
\omega_s\geqslant c\omega
\end{equation}
for some constants $c_0$ and $c$ depending on $t_0$ and $T$. Combining this with $-\sqrt{-1}\partial\bar{\partial}\log|s|_h^2=\theta$, we conclude that there exists a constant $C(t_0,T)$ depending on $t_0$ and $T$ such that
\begin{equation}\label{0402010}
-\hat{\Delta}\log|s|_h^2=\int_0^1\tr_{\omega_s}\theta ds\leq C(t_0,T)
\end{equation}
in $X\setminus D$. Then 
\begin{equation}\label{0402011}
  \frac{\partial \Phi(t)}{\partial t}\leqslant\hat{\Delta}\Phi(t)+aC(t_0,T).
\end{equation}

Let $\tilde{\Phi}=\Phi(t)-atC(t_0,T)-\epsilon t$, and $(t_1,x_1)$ be it maximum point. Since $\Phi(t)$ tends to $-\infty$ as $x\to D$ and so the space maximum point $x_1$ of $\tilde\Phi(t)$ in $[0,\frac{1}{2l}-t_0]\times (X\setminus D)$ attained away from $D$. If $t_1>0$, by the maximum principle, at $(t_1,x_1)$, we have
\begin{equation}\label{0402012}
 0\leqslant \left(\frac{\partial }{\partial t}-\hat{\Delta}\right)\tilde{\Phi}(t)\leqslant -\epsilon,
\end{equation}
which is impossible, hence $t_1=0$. Since
\begin{equation}\label{0402015}
\tilde\Phi(0,x_1)=\Phi(0,x_1)\leqslant M(0,x_1)-\phi(t_0,x_1)=0,
\end{equation}
then for $(t,x)\in [0,\frac{1}{2l}-t_0]\times X$, we obtain
\begin{equation}\label{0402013}
M(t)-\phi(t+t_0)\leqslant aTC(t_0,T)+\epsilon T-a\log|s|_h^2.
\end{equation}
We let $\epsilon\to0$ and then $a\rightarrow0$, there holds
\begin{equation}\label{0402014}
\begin{split}
\phi(t+t_0)&\geqslant(1-2lt)\phi(t_0)+2ltt_0\psi_o+tu_{t_0}+n(t\log t-t)-Ct\\
&=\phi(t_0)-2lt\left(\phi(t_0)-t_0\psi_o\right)+t\psi_o+t\left(u_{t_0}-\psi_0\right)+n(t\log t-t)-Ct\\
&\geqslant\phi(t_0)+t\psi_o-Ct+n(t\log t-t)
\end{split}
\end{equation}
for some uniform constant $C$ depending only on $\|\phi_0\|_{L^\infty(X)}$, $n$, $\omega$, $\theta$, $\eta$ and $T$. 

{\bf Claim.} For any $t\in(0,T]$, there holds
\begin{equation}\label{040201899}
\lim\limits_{s\to0^+}\left\|\left(\phi(t+s)-(t+s)\psi_o\right)-\left(\phi(t)-t\psi_o\right)\right\|_{L^\infty(X\setminus D)}=0.
\end{equation}
Since by the last condition \eqref{040600188} in Definition \ref{0510002}, there exists a uniform constant $C$ such that
\begin{equation}\label{0402016}
-C+a\log (t+\lambda)\leqslant\dot{\phi}(t+\lambda)-\psi_o\leqslant \frac{C}{t+\lambda}\ \ \ \ \text{for}\ \ \lambda\in[0,T-t]
\end{equation}
with $0<t<T$, by integrating above inequality from $0$ to $s$ with respect to $\lambda$ on both sides, we have
\begin{equation}\label{0402017}
\phi(t+s)-(t+s)\psi_o-\left(\phi(t)-t\psi_o\right)\leqslant C\left(\log\left(t+s\right)-\log t\right)
\end{equation}
and
\begin{equation}\label{040201788}
-Cs+a\left(\left(t+s\right)\log\left(t+s\right)-\left(t+s\right)\right)-a\left(t\log t-t\right)\leqslant\phi(t+s)-(t+s)\psi_o-\left(\phi(t)-t\psi_o\right)
\end{equation}
in $[0,T-t]\times (X\setminus D)$. Let $s\to0$ in \eqref{0402017} and \eqref{040201788}, for fix $0<t<T$, there holds
\begin{equation}\label{0402018}
\lim\limits_{s\to0^+}\left\|\left(\phi(t+s)-(t+s)\psi_o\right)-\left(\phi(t)-t\psi_o\right)\right\|_{L^\infty(X\setminus D)}=0.
\end{equation}
We complete the proof of the Claim.

By using \eqref{0402014}, we also have
\begin{equation}\label{0402019}
\phi(t+t_0)-(t+t_0)\psi_o\geqslant\phi(t_0)-t_0\psi_o-Ct+n(t\log t-t)\geqslant\phi(t_0)-Ct_0-Ct+n(t\log t-t),
\end{equation}
where we use the fact $\psi_\gamma$ decreasing to $\psi_o$ as $\gamma\to0$ and $\psi_\gamma$ is a continuous function on $X$ for $\gamma>0$. By using \eqref{033100888} and \eqref{0402018}, we have
\begin{equation}\label{0402020}
\phi_0+n(t\log t-t)-Ct\leqslant \phi(t)-t\psi_o\ \ \ \text{in}\ \ \ X\setminus D
\end{equation}
for all $t\in(0,\frac{1}{2l})$ after letting $t_0\to0$ in \eqref{0402019}.
\end{proof}
Next, we prove a maximum principle for the equation with cusp singualrity.
\begin{thm}\label{0402021}
Assume that $\varphi(t)$ and $\phi(t)$ are weak solutions to equation \eqref{PC} with initial values $\varphi_0$ and $\phi_0$ on $(0,T)\times X$ respectively, where $\varphi_0,\ \phi_0\in{\rm PSH}(X,\omega)\cap L^\infty(X)$.Then 
\begin{equation}\label{0402022}
\varphi(t)-\phi(t)\leqslant\sup\limits_X\left(\varphi_0-\phi_0\right)\ \ \ \text{in}\ \ \ X\setminus D
\end{equation}
for all $t\in[0,T)$.
\end{thm}
\begin{proof}
Fix $t_0>0$, we first consider 
\begin{equation}\label{0402023}
u(t)=\phi(t+t_0)+C_{t_0},
\end{equation}
where $C_{t_0}=Ct_0-n(t_0\log t_0-t_0)$ comes from \eqref{0401006}. Hence $u(t)$ is a solution to equation
\begin{equation}\label{0402024}
\frac{\partial u(t)}{\partial t}=\frac{\partial \phi(t+t_0)}{\partial t}=\log\frac{\left(\omega+(t+t_0)\nu+\dd u(t)\right)^n}{\omega^n}+h+\log|s|_h^2
\end{equation}
with initial value
\begin{equation}\label{0402025}
u(0)=\phi(t_0)+C_{t_0}\geqslant\phi_0+t_0\psi_o.
\end{equation}
On the other hand, 
\begin{equation}\label{0403001}
v(t)=\varphi(t+t_0)
\end{equation}
satisfies equation
\begin{equation}\label{0402026}
\frac{\partial v(t)}{\partial t}=\log\frac{\left(\omega+(t+t_0)\nu+\dd v(t)\right)^n}{\omega^n}+h+\log|s|_h^2
\end{equation}
with initial value $\varphi(t_0)$. If we denote
\begin{equation}\label{0402027}
\hat{\Delta}=\int_0^1 g_{s}^{i\bar{j}}\frac{\partial^2}{\partial z^i\partial\bar{z}^j}ds,
\end{equation}
where $g_s$ is the metric corresponding to $\omega_s=\omega_{(t+t_0)}+s\dd v(t)+(1-s)\dd u(t)$. Then 
\begin{equation}\label{0402028}
\Phi(t)=v(t)-u(t)+at_0e^{-\frac{2C}{t_0^2}}\log|s|_h^2
\end{equation}
evolves along the following equation
\begin{equation}\label{0402029}
  \frac{\partial \Phi(t)}{\partial t}\leqslant\hat{\Delta}\Phi(t)-at_0e^{-\frac{2C}{t_0^2}}\hat{\Delta}\log|s|_h^2.
\end{equation}
Since $\omega_{v(t)}$ and $\omega_{u(t)}$ are bounded from below by $e^{-\frac{C}{t_0^2}}\omega_o$ and $\omega_o\geqslant \frac{1}{2}\omega$, we obtain
\begin{equation}\label{0402030}
\omega_s\geqslant e^{-\frac{C}{t_0^2}}\omega.
\end{equation}
Combining this with $-\dd\log|s|_h^2=\theta$, we conclude that there exists a uniform constant $C$ depending on $\theta$ and $\omega$ such that
\begin{equation}\label{0405007}
-\hat{\Delta}\log|s|_h^2=\int_0^1\tr_{\omega_s}\theta ds\leq Ce^{\frac{C}{t_0^2}}
\end{equation}
in $X\setminus D$. Then 
\begin{equation}\label{0402032}
  \frac{\partial \Phi(t)}{\partial t}\leqslant\hat{\Delta}\Phi(t)+aCe^{-\frac{C}{t_0^2}}.
\end{equation}

Let $\tilde{\Phi}=\Phi(t)-aCe^{-\frac{C}{t_0^2}}t-\epsilon t$, and $(t_1,x_1)$ be its maximum point on $[0, T]$. Since
\begin{equation}\label{040202888}
\begin{split}
\Phi(t)&=v(t)-u(t)+at_0e^{-\frac{2C}{t_0^2}}\log|s|_h^2\\
&=\varphi(t+t_0)-\left(t+t_0\right)\psi_o-\left(\phi(t+t_0)-\left(t+t_0\right)\psi_o\right)-C_{t_0}+at_0e^{-\frac{2C}{t_0^2}}\log|s|_h^2
\end{split}
\end{equation}
and
\begin{equation}\label{}
\|\varphi(t+t_0)-\left(t+t_0\right)\psi_o\|_{L^\infty(X)}\ \ \ \text{and}\ \ \ \|\phi(t+t_0)-\left(t+t_0\right)\psi_o\|_{L^\infty(X)}
\end{equation}
are bounded, $\Phi(t)$ and so $\tilde\Phi(t)$ converge to $-\infty$ as $x$ tends to $D$. Hence $x_1$ must be in $X\setminus D$. If $t_1>0$, by the maximum principle, at $(t_1,x_1)$, we have
\begin{equation}\label{0402033}
 0\leqslant \left(\frac{\partial }{\partial t}-\hat{\Delta}\right)\tilde{\Phi}(t)\leqslant -\epsilon,
\end{equation}
which is impossible, hence $t_1=0$. Since by Proposition \ref{0401005},
\begin{equation}\label{0402034}
\begin{split}
\tilde\Phi(0,x_1)&\leqslant v(0,x_1)-u(0,x_1)+at_0e^{-\frac{2C}{t_0^2}}\log|s|_h^2(x_1)\\
&=\varphi(t_0,x_1)-\phi(t_0,x_1)-C_{t_0}+at_0e^{-\frac{2C}{t_0^2}}\log|s|_h^2(x_1)\\
&\leqslant\varphi(t_0,x_1)-\phi_0(x_1)-t_0\psi_o(x_1)+at_0e^{-\frac{2C}{t_0^2}}\log|s|_h^2(x_1)\\
&\leqslant\varphi(t_0,x_1)-\phi_0(x_1)+t_0\log\left(|s|_h^{2ae^{-\frac{2C}{t_0^2}}}\log^2|s|_h^2\right)(x_1)\\
&\leqslant\sup\limits_{X}\left(\varphi(t_0)+t_0\log\left(|s|_h^{2ae^{-\frac{2C}{t_0^2}}}\log^2|s|_h^2\right)-\phi_0\right)
\end{split}
\end{equation}
Then for $(t,x)\in [\delta,T]\times (X\setminus D)$, we obtain
\begin{equation}\label{0402035}
\begin{split}
\varphi(t+t_0)-\phi(t+t_0)&\leqslant\sup\limits_{X}\left(\varphi(t_0)+t_0\log\left(|s|_h^{2ae^{-\frac{2C}{t_0^2}}}\log^2|s|_h^2\right)-\phi_0\right)\\
&\ \ \ +aCTe^{-\frac{C}{t_0^2}}+\epsilon T+C_{t_0}-at_0e^{-\frac{2C}{t_0^2}}\log|s|_h^2.
\end{split}
\end{equation}
By using Claim \eqref{040201899},
\begin{equation}\label{0510007}
\begin{split}
\varphi(t+t_0)-\phi(t+t_0)&=\varphi(t+t_0)-\left(t+t_0\right)\psi_o-\left(\phi(t+t_0)-\left(t+t_0\right)\psi_o\right)\\
&\xrightarrow{t_0\rightarrow 0}\varphi(t)-t\psi_o-\left(\phi(t)-t\psi_o\right)\\
&=\varphi(t)-\phi(t).
\end{split}
\end{equation}
We let $\epsilon\to0$ and then $t_0\rightarrow0$, since $\varphi(t_0)+t_0\log\left(|s|_h^{2ae^{-\frac{2C}{t_0^2}}}\log^2|s|_h^2\right)$ is a $3\omega_{o}$-psh function which converges to $\varphi_0$ in $L^1$-sense as $t_0\to0$ on $X$, by Hartogs' Lemma, there holds
\begin{equation}\label{0402036}
\varphi(t)-\phi(t)\leqslant\sup\limits_{X}\left(\varphi_0-\phi_0\right).
\end{equation}
We complete the proof.
\end{proof}
As a corollary of above uniqueness and Theorem \ref{0330018}, we obtain the following uniqueness theorem for the twisted cusp K\"ahler-Ricci flow with initial metric.
\begin{thm}\label{0318007}
Assume that $\varphi_0\in{\rm PSH}(X,\omega)\cap L^\infty(X)$. Then there exists a unique weak solution to equation \eqref{PC} with initial value $\varphi_0$, or equivalently, there exists a unique weak solution to the twisted cusp K\"ahler-Ricci flow \eqref{PK} on $(0,T^\eta_{\max})\times X$.
\end{thm}
\begin{rem}\label{0510010}
In \cite{LSZ}, by approximating $\log|s|_h^2$ by $\log(\ve^2+|s|_h^2)$ as $\ve$ decreases to zero, Li-Shen-Zheng constructed a sequence of solutions $\phi_\ve(t)$ converging to a maximal weak solution $\phi_{\max}(t)$ (in the sense of potential) to the twisted cusp K\"ahler-Ricci flow \eqref{PK} on $(0,T^\eta_{\max})\times X$ due to $\log|s|_h^2\leqslant\log(\ve^2+|s|_h^2)$. 

In this paper, we obtain a weak solution $\varphi(t)$ to the twisted cusp K\"ahler-Ricci flow \eqref{PK} on $(0,T^\eta_{\max})\times X$ by taking a limit on $\varphi_\gamma(t)$ as $\gamma$ tends to zero. Moreover, we can also prove that $\varphi_\gamma(t)+t\gamma\log|s|_h^2$ increases to a minimal weak solution $\phi_{\min}(t)$ to the twisted cusp K\"ahler-Ricci flow \eqref{PK} on $(0,T^\eta_{\max})\times X$. Therefore,
\begin{equation}\label{0510011}
\varphi_\gamma(t)+t\gamma\log|s|_h^2\leqslant\phi_{\min}(t)=\varphi(t)\leqslant\phi_{\max}(t)\leqslant\phi_\ve(t).
\end{equation}
In fact, by using above uniqueness theorem Theorems \ref{0402021} and \ref{0318007}, we conclude that the solutions obtained from above three different methods are the same, that is,
\begin{equation}\label{0510012}
\phi_{\min}(t)=\varphi(t)=\phi_{\max}(t).
\end{equation}
\end{rem}
Combining Theorem \ref{0506001} and the uniqueness Theorem \ref{0318007}, we give the characterization of the limit flow of the twisted conical K\"ahler-Ricci flow as its cone angle tends to zero, which is exactly Theorem \ref{0416001}.

\end{document}